\documentclass[a4paper,12pt,leqno]{amsart}

\usepackage{amsmath,amssymb,amsxtra,comment,graphicx,psfrag}
\usepackage{bm,mathrsfs}
\usepackage{mathtools}
\usepackage{color}

\usepackage{enumerate,enumitem}

\usepackage{hyperref}

\usepackage{pdfsync}

\usepackage{hhline}
\usepackage[margin=1.25in]{geometry}

\allowdisplaybreaks

\definecolor{blue}{rgb}{0,0.0,0.9}
\definecolor{red}{rgb}{0.8,0.0,0}

 at 8. truept

\def\R{{\mathbb R}}

\def\N{{\mathbb N}}

\def\i{{\rm i}}
\def\tr|{|\!|\!|}

\DeclareMathOperator\Real {Re}

\theoremstyle{plain}
\newtheorem{theorem}{Theorem}[section]
\newtheorem{lemma}[theorem]{Lemma}
\newtheorem{proposition}[theorem]{Proposition}
\newtheorem{corollary}[theorem]{Corollary}

\theoremstyle{definition}

\newtheorem{example}{Example}[section]

\newtheorem{remark}{Remark}[section]

\numberwithin{equation}{section}
\numberwithin{table}{section}
\numberwithin{figure}{section}

\begin{document}

\title[Maximum norm analysis of implicit--explicit BDF methods]
{Maximum norm analysis of implicit--explicit\\
backward difference formulas\\ for nonlinear parabolic 
equations} 

\author[Georgios Akrivis]{Georgios Akrivis$\,\,$}
\address{Department of Computer Science \& Engineering, University of Ioannina, 451$\,$10
Ioannina, Greece} 
\email {\href{mailto:akrivis@cse.uoi.gr}{akrivis{\it @\,}cse.uoi.gr}} 

\author[Buyang Li]{$\,\,$Buyang Li $\,$}
\address{Department of Applied Mathematics, 
The Hong Kong Polytechnic University, Kowloon, Hong Kong.
} 
\email {\href{mailto:buyang.li@polyu.edu.hk}{buyang.li{\it @\,}polyu.edu.hk}}




\date{\today}


\begin{abstract} 
We establish optimal order a priori error estimates for implicit--explicit BDF methods for abstract 
semilinear parabolic equations with time-dependent operators in a complex Banach space settings, 
under a sharp condition on the non-self-adjointness of the linear operator. 
Our approach relies on the discrete maximal parabolic regularity of implicit BDF schemes 
for autonomous linear parabolic equations, recently established in \cite{KLL}, and on ideas from
\cite{ALL}. We illustrate the applicability of our results to four initial and boundary value problems,
namely two for second order, one for fractional order, and one for fourth order, namely
the  Cahn--Hilliard,  parabolic equations.
\end{abstract}
\maketitle

\section{Introduction}\label{Se:intro}
Let $V\hookrightarrow H=H'\hookrightarrow V'$ be a Gelfand triple of complex Hilbert spaces 
such that the restriction of the duality pairing $\langle\cdot,\cdot\rangle$ between 
$V$ and $V'$ to $V\times H$ coincides with the inner product 
$(\cdot,\cdot)$ on $H$. Let $T >0$ and consider an abstract initial value problem for a 
possibly nonlinear parabolic equation,
\begin{equation}
\label{ivp}
\left \{
\begin{aligned} 
&u' (t) + A(t)u(t)=B(t,u(t)), \quad 0<t<T,\\
&u(0)=u_0;
\end{aligned}
\right .
\end{equation}
here $A(t): V\to V'$ are bounded linear operators, while $B(t,\cdot) : V\cap W\to V'$ are nonlinear operators defined in the intersection of $V$ 
with another Banach space $W$. In this paper, we study the stability of the 
implicit--explicit BDF methods for the time discretization of \eqref{ivp}
and derive optimal order a priori error estimates.

Examples of the abstract problem \eqref{ivp} include (but are not restricted to) the following 
types of nonlinear parabolic partial differential equations.  

\begin{example}\label{Exmp1}
Consider the following initial and boundary value problem in a bounded domain 
$\varOmega\subset\R^d$, with smooth  boundary $\partial \varOmega$,
\begin{equation}\label{eq:Exmp1} 
\left \{
\begin{alignedat}{3} 
&\frac{\partial u}{\partial t} 
-\nabla\cdot\big((a(x,t)+\i\, b(x,t)) \nabla u \big)\\
&\qquad\,\, 
=f(u,x,t) +\nabla\cdot \bm{g}(u,x,t)\,\,\, 
&&\text{in}\,\,&&\varOmega\times(0,T), \\
&u=0 &&\text{on}\,\,&&\partial\varOmega\times(0,T), \\
&u(\cdot,0)=u_0 &&\text{in}\,\,&&\varOmega,
\end{alignedat}
\right .
\tag{I}
\end{equation}
where $a(x,t)>0$ and $b(x,t)$ are smooth real-valued functions in $\overline\varOmega\times[0,T]$, 
and the functions $f$ and $\bm{g}$ are only locally Lipschitz continuous with respect to $u$. 
For example, $f(v,x,t)=-v^3$ and $\bm{g}(v,x,t)=(e^v,\dotsc,0)$ are allowed. In this case, 
we have $V=H^1_0(\varOmega)$, $H=L^2(\varOmega)$ and $W=L^\infty(\varOmega)$. 
Then, the operators $B(t,v):=f(v,x,t) +\nabla\cdot \bm{g}(v,x,t) $ are well defined as nonlinear 
maps from $V\cap W$ to $V'$. 
\end{example}

\begin{example}\label{Exmp2}
Consider the following initial and boundary value problem in a bounded 
domain $\varOmega\subset\R^d$, with smooth boundary $\partial \varOmega$,
this time with stronger nonlinearity,  
\begin{equation}\label{eq:Exmp2} 
\left \{
\begin{alignedat}{3} 
&\frac{\partial u}{\partial t} 
-\nabla\cdot\big((a(x,t)+\i\, b(x,t)) \nabla u \big)\\
&\qquad\,\, =f(u,\nabla u,x,t) +\nabla\cdot \bm{g}(u,\nabla u,x,t)\,\,\,
&&\text{in}\,\,&&\varOmega\times(0,T), \\
&u=0 &&\text{on}\,\,&&\partial\varOmega\times(0,T), \\
&u(\cdot,0)=u_0 &&\text{in}\,\,&&\varOmega,
\end{alignedat}
\right .
\tag{II}
\end{equation}
where $a$ and $b$ are as in \eqref{eq:Exmp1}  while the functions $f$ and $\bm{g}$ are only locally Lipschitz 
continuous with respect to $u$ and  $\nabla u$. For instance, $f(v,\nabla v,x,t)=-|\nabla v|^4v $ and 
$\bm{g}(v,\nabla v,x,t)=|\nabla v|^4\nabla v$ are allowed. In this case, we have $V=H^1_0(\varOmega)$,
 $H=L^2(\varOmega)$ and $W=W^{1,\infty}(\varOmega)$. Then, the operators 
 $B(t,v):=f(v,\nabla v,x,t) +\nabla\cdot \bm{g}(v,\nabla v,x,t) $ are well defined as nonlinear 
 maps from $V\cap W$ to $V'$. 
\end{example}

\begin{example}\label{Exmp3}
Consider the Cauchy problem for a fractional partial differential equation in $\R^d$, $d\ge 1$, 
\begin{equation}\label{eq:Exmp3} 
\left \{
\begin{alignedat}{3} 
&\frac{\partial u}{\partial t} 
+(-\varDelta)^{1/2} u =f(u)  \,\,\,
&&\text{in}\,\,&& \R^d\times(0,T), \\
&u(\cdot,0)=u_0 &&\text{in}\,\,&& \R^d ,
\end{alignedat}
\right .
\tag{III}
\end{equation}
with $f$ a given smooth function of $u$  
such that $f(0)=0.$ 
For example, $f(u)=e^u-1$. In this case, we have $V=H^{\frac{1}{2}}(\R^d)$, $H=L^2(\R^d)$ 
and $W=L^2(\R^d)\cap L^{\infty}(\R^d)$. Then, the operators 
$B(t,v):=f(v) $ are well defined as nonlinear maps from  $V\cap W$ to $V'$.
\end{example}

\begin{example}\label{Exmp4}
Consider the Cauchy problem for the Cahn--Hilliard equation in $\R^d$, $d\ge 1$, 
\begin{equation}\label{eq:Exmp4} 
\left \{
\begin{alignedat}{3} 
&\frac{\partial u}{\partial t} +\varDelta^2u =\varDelta f(u) \,\,\,
&&\text{in}\,\,&&\R^d\times(0,T), \\
&u(\cdot,0)=u_0 &&\text{in}\,\,&&\R^d ,
\end{alignedat}
\right .
\tag{IV}
\end{equation}
with $f$ a given smooth function of $u$, such as $f(u)=u^3-u$; 
see \cite{CR12}. In this case, we have $V=H^2(\R^d)$,
 $H=L^2(\R^d)$ and $W=H^2(\R^d)\cap W^{2,\infty}(\R^d).$
Then, the operators $B(t,v):=\varDelta f(u) =f'(u)\varDelta u+f''(u)|\nabla u|^2 $ are well 
defined as  nonlinear maps from $V\cap W$ to $V'$. 
\end{example}

For $k=1,\dotsc,6,$ consider the implicit $k$-step BDF method $(\delta,\beta)$ and the
explicit $k$-step method $(\delta,\gamma)$ described by the polynomials $\delta, \beta$ 
and $\gamma,$ 
\begin{equation}
\label{BDF1}
\left\{
\begin{aligned}
\delta (\zeta)&{}= \sum_{\ell=1}^k \frac 1\ell  (1-\zeta)^\ell
=\sum\limits^k_{i=0}\delta_i \zeta ^{i}, \quad \beta(\zeta)= 1,\\ 
\gamma (\zeta)&{}=\frac 1\zeta\big [1- (1-\zeta)^k\big ] 
=\sum_{i=0}^{k-1} \gamma_i\zeta^i.
\end{aligned}
\right .
\end{equation}
The BDF method $(\delta,\beta)$ is known to have order $k$ and to be A$(\alpha_k)$-stable
with angles $\alpha_1=\alpha_2=90^\circ, \alpha_3=86.03^\circ, \alpha_4=73.35^\circ,
\alpha_5=51.84^\circ$ and $\alpha_6=17.84^\circ$; see \cite[Section V.2]{HW}. 
A$(\alpha)$-stability is equivalent to $|\arg \delta (\zeta)|\le \pi -\alpha$ 
for $|\zeta|\le 1.$ Note that the first- and second-order BDF methods are 
A-stable, that is $\Real \delta (\zeta)\ge 0$ for $|\zeta|\le 1.$
For a given polynomial $\delta,$ the scheme $(\delta,\gamma)$ is the unique explicit 
$k$-step  scheme of order $k;$ the order of all other explicit $k$-step schemes 
$(\delta,\tilde \gamma)$ is at most $k-1.$

Let $N\in \N, N\ge k,$ and consider  a uniform partition $t_n:=n\tau, n=0,\dotsc,N,$ of the interval  $[0,T],$ with time step $\tau:=T/N.$ 
Since the nonlinear operators $B(t,\cdot)$ on the right-hand side of \eqref{ivp} is only defined on $V\cap W$, 
we shall choose some Banach spaces 
$D\subset V\cap W$ and $X$ such that 
\begin{align}\label{DWXAB}
\begin{array}{ccccc}
V & \subset & H & \subset & V' \\
\cup &         &   &             & \cup \\
D &\subset & W  &\subset & X
\end{array}
\qquad\text{and}\qquad
\begin{array}{llllll}
A(t): &\!\!\! D\rightarrow X , \\[6pt]
B(t,\cdot): &\!\!\! D\rightarrow X , 
\end{array} 
\end{align}
and assume that  we are given starting approximations $u_0,\dotsc,u_{k-1}\in D,$ 
to the nodal values $u^\star_j:=u(t_j)$, $j=0,\dotsc,k-1.$  
We discretize \eqref{ivp} in time by the implicit--explicit $k$-step 
BDF method $(\delta,\beta,\gamma)$, i.e., we define approximations $u_m\in D$ 
to the nodal values $u^\star_m:=u(t_m)$  of the exact solution as follows
%
\begin{equation}
\label{abg}
\frac{1}{\tau}\sum\limits^k_{i=0}\delta_iu_{n-i} +  A(t_n)u_n= 
\sum\limits^{k-1}_{i=0}\gamma_iB(t_{n-i-1},u_{n-i-1}),
\quad n=k,\dotsc,N .
\end{equation}
In other words, the linear part $A(t)u(t)$ of the equation in \eqref{ivp} is discretized by the
implicit BDF scheme $(\delta,\beta),$  while the nonlinear part $B(t,u(t))$  is discretized by the
explicit BDF scheme $(\delta,\gamma).$
As a result, the unknown $u_n$ appears only on the left-hand side of the implicit--explicit
BDF scheme \eqref{abg}; 
therefore, to advance in time, one only needs to solve one linear equation, which reduces to 
a linear system if one discretizes also in space, at each time level. 

Motivated by Examples \ref{Exmp1}--\ref{Exmp4}, we only require that the nonlinear operators 
$B(t,\cdot)$ are Lipschitz continuous in a tube $T_{u,r}^D,$ 
\begin{equation}\label{tube-W}
T_{u,r}^D:=\{v\in D: \min_{0\le t\le T}\|v-u(t)\|_W\le  r\},
\end{equation}
around the solution $u$, uniformly in $t,$ where $W$ may be a suitably chosen 
$L^\infty$-based Sobolev space in practical applications, 
such as $L^\infty (\varOmega)$ or $W^{1,\infty} (\varOmega),$ depending on the type of the nonlinearity.   
The main difficulty in numerical analysis of such problems is that one has to 
prove uniform boundedness of the numerical solutions $u_n, n=k,\dotsc,N$, 
with respect to the norm of $W$. To overcome this difficulty, we study the stability of the 
implicit--explicit BDF methods for \eqref{ivp} in a Banach space setting, 
by using the mathematical tool of discrete maximal $L^p$-regularity. 
In contrast to the present approach, in \cite{A-SINUM,ACM2,AL,AK}  the local Lipschitz condition 
was imposed in tubes $T_{u,r}^V,$
\begin{equation}\label{tube-V}
T_{u,r}^V:=\{v\in V: \min_{0\le t\le T}\|v-u(t)\|_V\le  r\} ,
\end{equation}
defined in terms of the norm $\|\cdot\|_V;$ as a consequence, the analysis of \cite{ACM2,A-SINUM,AL,AK} 
is not directly applicable to Examples \ref{Exmp1}--\ref{Exmp4}, if we only 
consider the discretization  in time, since it cannot ensure that the approximations are sufficiently 
close to the exact solution in the norm $\|\cdot\|_W;$ 
it is, however, applicable, usually under mild mesh-conditions, in the fully discrete case, i.e., 
if we combine the time stepping schemes with discretization in space;  cf., e.g., \cite{ACM2}. 
The present analysis allows us to avoid  
growth conditions on the nonlinearities.

Our approach is based on the discrete  maximal parabolic regularity property of the
implicit BDF methods. Let us briefly recall the relevant definitions:
An elliptic differential operator $-A$ on a Banach space $\big (X,\|\cdot\|_X\big )$
has maximal $L^p$-regularity, $1<p<\infty,$ if the solution $u$ of the initial value problem
\begin{equation}\label{De:max-reg1}
u' (t) + Au(t)=f(t), \quad 0<t<T,\quad u(0)=0,
\end{equation}
with forcing term $f\in L^p  (0,T;X ),$ satisfies the a priori estimate
\begin{equation}\label{De:max-reg2}
\int_0^T \|u' (t)\|_X^p\, dt+\int_0^T \|Au(t)\|_X^p\, dt\le C\int_0^T \|f (t)\|_X^p\, dt
\end{equation}
with some constant $C.$ In other words, if both terms $u'$ and $Au$ on the left-hand
side of the autonomous parabolic equation are well defined and have the same 
(i.e., maximal) regularity  as the forcing term $f.$  It is well known that if an operator
has maximal $L^p$-regularity for some $1<p<\infty,$ then it has maximal $L^p$-regularity 
for all $1<p<\infty.$ Replacing $A$ by $A(t)$ both in 
\eqref{De:max-reg1} and \eqref{De:max-reg2}, the definition extends to non-autonomous 
parabolic equations with a family of elliptic differential operator $-A(t), t\in [0,T],$ on $X$ 
with the same domain, $D=D(A(t)), t\in [0,T].$ Maximal regularity is an important 
tool in the theory of nonlinear parabolic equations. 
For an excellent account of the maximal  regularity theory,
in particular, for the important Weis' characterization on UMD spaces (which include
$L^q(\varOmega), 1<q<\infty$), and for relevant references, we refer to the lecture notes by Kunstmann \& Weis \cite{KW}. 
Space-discrete analogues of the maximal parabolic regularity, uniform in the spatial mesh size, can be found in \cite{Geiss1,Geiss2,Li15,LS15}.

For the time-discrete maximal parabolic regularity property of autonomous parabolic
equations, uniformly in the timestep,  we refer to \cite{KLL} and the references therein. 
The main result of \cite{KLL} is that A-stable Runge--Kutta methods, satisfying  minor 
additional conditions, such as Gauss--Legendre and Radau IIA methods, as well as
one- and two-step BDF methods preserve maximal regularity; it is also shown in \cite{KLL}
that higher-order $k$-step BDF methods, $k=3,\dotsc,6,$ preserve maximal regularity 
under some natural additional conditions on the operator accounting for the lack of 
A-stability of these methods.

Here we establish local stability of the implicit--explicit BDF methods \eqref{abg} for  \eqref{ivp}
under smallness conditions on the \emph{stability constants} $\lambda,$ which is equal to
$1$ in the case of self-adjoint operators while $\lambda-1$ may be viewed
as a measure of the non-self-adjointness of the linear operators $A(t)$ (see \eqref{bounded-a}
in the sequel),  for $k=3,\dotsc,6,$ and  $\tilde \lambda,$ the constant in the local Lipschitz condition 
on the nonlinear operators $B(t,\cdot), t\in [0,T],$ (see  \eqref{Lipschitz} in the sequel).
While we can quantify the condition on the stability constant $\lambda,$
actually in a sharp way, unfortunately we cannot quantify the condition on $\tilde \lambda,$
since we have no control on the constant  $C$ in the discrete maximal regularity 
of the implicit  BDF schemes (see \eqref{BDF_MaxReg} in the sequel);
therefore, we shall assume that  $\tilde \lambda$ is sufficiently small,
in the case $k=3,\dotsc,6$ depending also on the value of $\lambda.$ 
In the applications, 
in case the differential operators $B(t,\cdot), t\in [0,T],$  are of order 
lower than the order of the linear differential operators $A(t),$
the Lipschitz constant $\tilde \lambda$ in the local Lipschitz condition \eqref{Lipschitz} 
can typically be chosen arbitrarily small; as we will see, this is, in particular, the case
for Examples \ref{Exmp1}, \ref{Exmp3},  and \ref{Exmp4}.

More precisely, we shall assume that $\lambda$ does not exceed $1/\cos \alpha_k,$
\begin{equation}
\label{intr-stab-abg}
\lambda<\frac{1}{\cos \alpha_k}.
\end{equation}
This  is actually a sharp condition on the non-self-adjointness of the linear operators $A(t),$ 
in the sense that if $\lambda$ exceeds the right-hand side in \eqref{intr-stab-abg}, then
the (implicit) $k$-step BDF method is in general unstable for the linear equation 
$u' (t) + A(t)u(t)=0.$ Indeed, for $k=1$ and $k=2$ condition \eqref{intr-stab-abg} is void, 
and, for $k=3,\dotsc,6,$ letting $\tilde A$ be a positive definite self-adjoint operator and 
considering the ``rotated'' operator $A:=e^{\i \varphi} \tilde A,$ we see that condition 
\eqref{bounded-a} is satisfied as an equality with $\lambda=1/\cos \varphi.$ But, the 
eigenvalues of $A$ are of the form $\rho e^{\i \varphi},$ with $\rho >0;$ now, 
for $\alpha_k<\varphi<\pi$ the eigenvalues are not included in the stability sector
$\varSigma_{\alpha_k}:=\{z=re^{\i\vartheta}: r\ge 0, |\vartheta|\le \alpha_k\}$ of the 
$k$-step BDF scheme and, therefore, according to the von 
Neumann stability criterion, the $k$-step BDF  method is not  unconditionally stable for the 
equation $u' (t) + e^{\i \varphi} \tilde A u(t)=0.$

Let us note that in the case of the linear operators of the differential equations 
in the initial and boundary value problems of Examples \ref{Exmp1} and \ref{Exmp2},
condition \eqref{intr-stab-abg} takes the form 
\begin{equation}
\label{intr-stab-abg2}
\lambda=\max_{\substack{x\in \bar \varOmega\\*[2pt] t\in [0,T]}}
\frac{|a(x,t)+\i\, b(x,t)|}{a(x,t)}<\frac{1}{\cos \alpha_k},
\end{equation}
which can also be equivalently written as
\begin{equation}
\label{intr-stab-abg3}
\max_{\substack{x\in \bar \varOmega\\*[2pt] t\in [0,T]}}
\frac{|b(x,t)|}{a(x,t)}<\tan \alpha_k.
\end{equation}

Implicit--explicit multistep methods, and in particular implicit--explicit BDF sche\-mes, 
were introduced and analyzed for non-autonomous linear parabolic equations in \cite{C}.
In a Hilbert space setting, implicit--explicit BDF methods can be analyzed by various
techniques, such as spectral and Fourier techniques 
(see, e.g., \cite{ACM2, A-MATHCOM}),
and energy methods (see, e.g., \cite{A-SINUM, AL, AK}); both techniques
have advantages and drawbacks. In a Banach space setting, the analysis
of implicit BDF methods for autonomous linear parabolic equations in \cite{KLL}
is based on maximal regularity, while the analysis
of implicit as well as of linearly implicit BDF methods for quasilinear parabolic equations 
with real symmetric coefficients in  \cite{ALL}  combines maximal regularity and energy techniques.
To our best knowledge, implicit--explicit BDF schemes 
for nonlinear parabolic equations, in particular with complex coefficients, 
have not been previously analyzed in a Banach space setting.

Error estimates under sharp stability conditions of the form \eqref{intr-stab-abg} 
were established in the Hilbert space setting by spectral and Fourier techniques in \cite{S} 
for implicit multistep methods, including BDF schemes, for linear  parabolic equations, and
in \cite{A-MATHCOM} for implicit--explicit multistep methods, including implicit--explicit
BDF schemes, for a class of nonlinear parabolic equations with linear operators of a 
special form; more precisely, for linear operators of the form considered in \eqref{eq:Exmp1} 
and \eqref{eq:Exmp2}, the assumption in  \cite{A-MATHCOM} is that $a$ is independent of $t$
and $b$ is of the form $b(x,t)=\tilde b(t) a(x).$

For further stability analyses of implicit multistep methods for autonomous 
linear parabolic equations in a Banach space setting, we refer to \cite{Pal} and 
references therein; in particular, in \cite{Pal} stability under the optimal condition  
\eqref{intr-stab-abg} is established.

An outline of the paper is as follows:  In Section \ref{Se:abstr} we present our 
abstract framework and discuss its applicability to the cases of the initial
and boundary value problems of Examples \ref{Exmp1}--\ref{Exmp4}.
In Section \ref{Sec:MaxReg} we establish discrete maximal regularity of 
BDF methods for non-autonomous linear parabolic equations, thus extending
recent results of \cite{KLL} concerning autonomous parabolic equations.
Section \ref{Sec:Angle} is devoted to the angle of analyticity and R-boundedness
of non-autonomous linear operators under assumptions \eqref{en:A1} and \eqref{en:A3};
a lower bound for this angle is given. Our main results are presented in Sections 
\ref{Se:stab} and \ref{Se:cons}: We first establish local 
stability of the implicit--explicit BDF schemes \eqref{abg} in Section \ref{Se:stab},
which is then combined with the consistency of the methods 
and leads to optimal order a priori error estimates in Section \ref{Se:cons}.
In Section \ref{Sec:ProofProposition} we verify the
applicability of our abstract framework to the initial and boundary value problems
\eqref{eq:Exmp1}--\eqref{eq:Exmp4} of Examples \ref{Exmp1}--\ref{Exmp4},
respectively, in the concrete spaces given in Propositions \ref{lem:framework}--\ref{lem:framework4}.

\section{Abstract framework and applications to Examples \ref{Exmp1}--\ref{Exmp4}}\label{Se:abstr}

In this section we present our abstract framework, which is, in particular,
applicable to Examples \ref{Exmp1}--\ref{Exmp4}.

For a sequence $(v_n)_{n=1}^N$ and a given stepsize $\tau$, we shall use the notation
\[ \big\|(v_n )_{n=1}^N\big\|_{L^p(X)} 
= \bigg( \tau \sum_{n=1}^N \| v_n \|_X^p \bigg)^{\!1/p},\]
which is the $L^p(0,N\tau;X)$ norm of the piecewise constant function taking the value~$v_n$ 
in the subinterval $(t_{n-1},t_n], n=1,\dotsc,N.$ 

For any two Banach spaces $X$ and $Y$ which are imbedded into a common Hausdorff topological 
space, we denote by $X\cap Y$ the Banach space consisting of elements in both $X$ and $Y$, 
equipped with the norm 
\[ \|v\|_{X\cap Y}:=\|v\|_X+\|v\|_Y .\]

We will work with the Banach space setting  under the following assumptions: 
\begin{enumerate}[label=(A\arabic*),ref=A\arabic*]\itemsep=5pt
\item  (\emph{Generation of bounded analytic semigroups by the linear operators})\label{en:A1} \\
For all $s\in[0,T]$, there exists a unique solution 
$v\in H^1(\R_+;V')\cap L^2(\R_+;V) \hookrightarrow C([0,\infty);H)$  of the initial value problem
\begin{equation}\label{HomoEq}
\left\{
\begin{aligned}
&v'(t)+A(s)v(t)=0,\quad t>0,\\
&v(0)=v_0 .
\end{aligned}
\right. 
\end{equation}
The solution map $E_s^H(t):H\rightarrow H$, which maps $v_0$ to $v(t)$,  
extends to a bounded analytic semigroup $\{E_s^H(z)\}_{z\in\varSigma_{\theta_s}}$ on $H$, 
where $\theta_s\in(0,\pi/2]$ is the maximal angle of analyticity (i.e., the supremum of all such angles).   
Moreover,  the domain $D(A_H(s))=D_H\hookrightarrow V$ of the generator $-A_H(s)$ of the 
semigroup  $\{E_s^H(z)\}_{z\in\varSigma_{\theta_s}}$ 
is supposed to be independent of $s\in[0,T]$ and compactly imbedded into $H$, with 
\[\theta:=\inf_{s\in[0,T]}\theta_s>0 . \]
\item (\emph{Discrete maximal regularity of the implicit BDF schemes})\label{en:A2} \\ 
There exist Banach spaces $D$ and $X$ satisfying \eqref{DWXAB}. 
Moreover, if $\theta>\pi/2-\alpha_k,$ then, for all $s\in[0,T],$ 
the $k$-step BDF solution determined by
\begin{equation}
\label{lin-eq-Av-bdf}
\frac1\tau \sum_{j=0}^k \delta_j v_{n-j} + A(s) v_n 
= f_n,\quad k\le n\le N ,
\end{equation}
with given $f_n\in X$ and given starting values $v_0,\dotsc,v_{k-1}\in D$, is bounded by
\begin{equation}\label{BDF_MaxReg}
\begin{aligned}
&\frac{1}{\tau}\big\|(v_n-v_{n-1} )_{n=k}^N\big\|_{ L^p(X)} 
+ \big\|(v_n )_{n=k}^N\big\|_{ L^p(D)} \\
&\le C\Big(
\big\|(f_n)_{n=k}^N\big\|_{ L^p(X)} 
+\frac{1}{\tau}\big\|(v_i)_{i=0}^{k-1}\big\|_{ L^p(X)} 
+ \big\|(v_i )_{i=0}^{k-1}\big\|_{ L^p(D)} \Big) , 
\end{aligned}
\end{equation}
where the constant $C$ is independent of $\tau$ and $s\in[0,T]$. 
\item (\emph{Boundedness and bounded variation of $A(t):D\rightarrow X,$} and coercivity of 
$A(t):V\rightarrow V'$)\label{en:A3} \\
There exist two positive constants $M_1$ and $M_2$ such that 
the operator norms $\|A(t)\|_{{\mathcal L}(D,X)}$ 
are uniformly bounded by $M_1,$ for all $t\in [0,T],$
and
\begin{equation}\label{bv-t}
\sum_{i=1}^m \| A(\tau_i)-A(\tau_{i-1})\|_{{\mathcal L}(D,X)} \le  M_2
\end{equation}
for any partition $0=\tau_0<\tau_1<\cdots<\tau_m=T$ of $[0,T].$ 
There exists a constant $\lambda\ge 1$ such that 
\begin{equation}
\label{bounded-a}
|\langle A(t)v,v\rangle| 
\le \lambda \, \Real \langle A(t)v,v\rangle \quad \forall \, v\in V .
\end{equation}
\item (\emph{Local Lipschitz continuity of the nonlinear operators 
$B(t,\cdot)$})\label{en:A4} \\
There exists a constant $r_0>0$ such that the operators $B(t,\cdot): D\to X$ 
satisfy a local Lipschitz condition in a tube $T_{u,r_0}^D;$ see \eqref{tube-W};
more precisely, there exist nonnegative constants $\tilde\lambda$ 
and $C_{B}$, independent of $t\in[0,T]$, such that, for all $v,w\in T_{u,r_0}^D$,
\begin{equation}
\begin{aligned}
\label{Lipschitz}
\| B(t,v)-B(t,w) \| _{X}
\le \tilde\lambda\|v-w\|_D+C_{B}(\|v\|_D+\|w\|_D)\| v-w \|_{W} .  
\end{aligned}
\end{equation}

\item  (\emph{Control of the $W$-norm by the maximal $L^p$-regularity})\label{en:A5} \\
For any $\varepsilon>0,$ there exists $C_\varepsilon>0$ such that 
\begin{equation}
\begin{aligned}
\label{Compact_interp}
\|v\|_{W}\le \varepsilon \|v\|_{D}+C_\varepsilon \|v\|_{X} \quad\forall\, v\in D .
\end{aligned}
\end{equation}
%
For some $1<p<\infty$, we have a time-space continuous imbedding 
$W^{1,p}(0,T;X)$ $ \cap L^p(0,T;D)\hookrightarrow L^\infty(0,T;W)$: 
there exists a positive constant 
$C_{W}$ such that,  for all $v \in W^{1,p}(0,T;X) \cap L^p(0,T;D)$ with $v(0)=0$,
\begin{equation}
\begin{aligned}
\label{MaxLp_interp}
\|v\|_{L^\infty(0,T;W)} \le C_{W}\big( \| v' \|_{L^p(0,T;X)} + \| v \|_{L^p(0,T;D)} \big).\end{aligned}
\end{equation}
\end{enumerate}

\begin{remark} 
{\upshape 
Similar assumptions in the Banach space setting 
\eqref{DWXAB} were recently used in \cite{ALL}. 
Assumptions \eqref{en:A1} and \eqref{en:A2} are now connected 
through $\theta$, the angle of analyticity, since we study parabolic equations with complex coefficients. 
Assumption \eqref{en:A3} is now relaxed to operators of bounded variation in time, possibly 
discontinuous.} 
\end{remark}
\begin{remark} 
{\upshape 
Assumption \eqref{en:A1} guarantees that $A_H(s)v=A(s)v$ for $v\in D_H\hookrightarrow V .$ 
In other words, $A(s)$ is an extension of the operator  $A_H(s)$. 
} 
\end{remark}

\eqref{en:A1}--\eqref{en:A5} are natural assumptions for studying many PDE problems, 
as can be seen in the following three Propositions.

\begin{proposition}[The abstract framework is 
applicable to Examples \ref{Exmp1}--\ref{Exmp2}] \label{lem:framework}
Let $q\in(d,\infty)\cap[2,\infty)$ and $p\in(1,\infty)$ be such that $2/p+d/q<1.$
Then, assumptions \eqref{en:A1}--\eqref{en:A5} are satisfied for the initial and boundary value problem \eqref{eq:Exmp1} in Example \ref{Exmp1}
with $D_H=H^2(\varOmega)\cap
 H^1_0(\varOmega),  D=W^{1,q}_0(\varOmega),  W=L^\infty(\varOmega)$ and 
 $X=W^{-1,q}(\varOmega)$, with $\lambda$ as in the left-hand side of \eqref{intr-stab-abg2} 
 and $\tilde\lambda=0$. 


\indent 
Let $q\in(d,\infty)$ and $p\in(1,\infty)$ be such that $2/p+d/q<1.$
Then,  assumptions \eqref{en:A1}--\eqref{en:A5} are satisfied for the initial and 
boundary value problem \eqref{eq:Exmp2} in Example \ref{Exmp2}
with $D_H=H^2(\varOmega)\cap H^1_0(\varOmega), 
D=W^{2,q}(\varOmega)\cap W^{1,q}_0(\varOmega), 
W=W^{1,\infty}(\varOmega)$ and $X=L^q(\varOmega)$, with $\lambda$ as 
in the left-hand side of \eqref{intr-stab-abg2} and
\begin{equation}
\label{tildelambda-Exam3}
\tilde\lambda=\sup_{t\in[0,T]}\sum_{i,j=1}^d\sup_{x\in\varOmega}
\sup_{\substack{|\xi-u(x,t)|\le r\\*[2pt] |\vec\eta-\nabla u(x,t)|\le r}}
\bigg| \frac{\partial g_i(\xi,\vec\eta,x,t)}{\partial\eta_j} \bigg| \, .
\end{equation}
\end{proposition}

\begin{proposition}[The abstract framework is 
applicable to Example \ref{Exmp3}] \label{lem:framework3} 
Let $q\in(d,\infty)\cap [2,\infty)$ and $p\in(1,\infty)$ be such that $1/p+d/q<1.$
 Then, assumptions \eqref{en:A1}--\eqref{en:A5} are satisfied for the Cauchy problem 
\eqref{eq:Exmp3} in Example \ref{Exmp3}
with  $D_H=H^1({\mathbb R}^d), D=H^1({\mathbb R}^d)\cap W^{1,q}({\mathbb R}^d)$, $X=L^2({\mathbb R}^d)\cap L^q({\mathbb R}^d)$ 
and $W=L^2({\mathbb R}^d)\cap L^{\infty}({\mathbb R}^d)\hookrightarrow X$, 
with the stability constants $\lambda=1$ and  $\tilde\lambda=0$. 
\end{proposition}

\begin{proposition} [The abstract framework is applicable to Example \ref{Exmp4}] \label{lem:framework4}
Let $q\in(d/2,\infty)\cap [2,\infty)$ and $p\in(1,\infty)$ be such that  $4/p+d/q<2.$
Then, assumptions \eqref{en:A1}--\eqref{en:A5} are satisfied for the Cauchy problem 
\eqref{eq:Exmp4} in Example \ref{Exmp4}
with  $D_H=H^4({\mathbb R}^d), D=H^4({\mathbb R}^d)\cap W^{4,q}({\mathbb R}^d)$, 
$X=L^2({\mathbb R}^d)\cap L^q({\mathbb R}^d)$ and 
$W=H^2({\mathbb R}^d)\cap W^{2,\infty}({\mathbb R}^d)\hookrightarrow X$, 
with the stability constants $\lambda=1$ and  $\tilde\lambda=0$. 
\end{proposition}

The proofs of Propositions \ref{lem:framework}--\ref{lem:framework4} will be given in Section 
\ref{Sec:ProofProposition}; the proofs of Propositions \ref{lem:framework3} and \ref{lem:framework4} 
are based on \cite[Corollary 2.7 and Proposition 2.9]{BK02} and \cite[Example 3.2 (A)]{HP97}, 
respectively.

In Sections \ref{Sec:MaxReg}--\ref{Se:cons}, we study the stability of the 
implicit--explicit BDF methods \eqref{abg} under assumptions \eqref{en:A1}--\eqref{en:A5}. 

\begin{remark}[Condition \eqref{bounded-a} expressed in terms of time-dependent 
norms]\label{Re:time-dep-n}
{\upshape
Following \cite{A-SINUM,AL}, we introduce time-dependent norms and rewrite
\eqref{bounded-a} in an equivalent way.
The time-dependent norms are based on the decomposition of the coercive operators 
$A(t)$ in their self-adjoint and  anti-self-adjoint parts $A_s(t)$ and $A_a(t),$ 
respectively,
\[A_s(t):=\frac 12 \big [A(t)+A(t)^\star\big ],\quad A_a(t):=
\frac 12 \big [A(t)-A(t)^\star\big ].\]
Now, we endow $V$ with the time-dependent norms $\|\cdot\|_t,$ 
\[\|v\|_t:=\langle A_s(t)v,v\rangle^{1/2} \quad \forall v\in V,\]
which are uniformly equivalent to the norm $\|\cdot\|_V,$
and denote by $\|\cdot\|_{\star,t}$ the corresponding dual norm on $V',$
\[\forall v\in V'\quad \|v\|_{\star,t}:=\sup_{w\in V\setminus\{0\}}\frac {|\langle v,w\rangle|}{\|w\|_t}=
\sup_{\substack{w\in V\\ \|w\|_t=1}}|\langle v,w\rangle|.\]
Then, condition \eqref{bounded-a} simple says that  the operators $A(t) : V\to V'$ 
are uniformly bounded and their norms do not exceed $\lambda,$
\begin{equation}
\label{bounded-a-new}
\| A(t)v\|_{\star,t} \le \lambda \| v\|_t \quad \forall v\in V.
\end{equation}
In the case of self-adjoint operators $A(t),$ \eqref{bounded-a} and \eqref{bounded-a-new}
are satisfied with $\lambda=1;$ otherwise $\lambda>1.$
}
\end{remark}

\begin{remark}[Equivalent form of  \eqref{BDF_MaxReg} in the case $v_0=0$]\label{Remark:A2}
{\upshape
If $v_0=0$, then \eqref{BDF_MaxReg} in \eqref{en:A2} can be equivalently written in a more 
symmetric form as 
\begin{equation}\label{BDF_MaxReg_v0}
\begin{aligned}
&\frac{1}{\tau}\big\|(v_n-v_{n-1} )_{n=k}^N\big\|_{ L^p(X)} 
+ \big\|(v_n )_{n=k}^N\big\|_{ L^p(D)} \\
&\le C\Big(\big\|(f_n)_{n=k}^N\big\|_{ L^p(X)} 
+\frac{1}{\tau}\big\|(v_i-v_{i-1} )_{i=1}^{k-1}\big\|_{ L^p(X)} 
+ \big\|(v_i )_{i=1}^{k-1}\big\|_{ L^p(D)} \Big) . 
\end{aligned}
\end{equation}
}
\end{remark}

\section{Discrete maximal $L^p$-regularity of BDF methods}\label{Sec:MaxReg}

We will work with the abstract assumptions \eqref{en:A1}--\eqref{en:A3} of the previous 
section and will show that the (implicit) BDF methods satisfy the discrete maximal  
parabolic regularity property, when applied to initial value problems of the form 
\eqref{ivp} for linear parabolic equations, i.e., with right-hand side $B(t,u(t))=f(t);$
this property is of independent interest and will also play a crucial role in our stability 
analysis of the implicit--explicit BDF methods in section \ref{Se:stab}.
 We thus extend the corresponding discrete maximal  parabolic regularity result \eqref{en:A2} 
for autonomous parabolic equations 
of \cite[Theorems 4.1--4.2]{KLL} to the case of non-autonomous equations, with operators
continuous with respect to time.

\begin{proposition}[Discrete maximal parabolic regularity]\label{Prop:MaxLp}
Under assumptions \eqref{en:A1}--\eqref{en:A3}, there exist positive constants $\tau_0$ and 
$C$ {\rm(}independent of $\tau$, but possibly depending on $T${\rm)} such that, for every 
stepsize $\tau\le\tau_0$,  the $k$-step BDF method, 
\begin{equation}
\label{lin-eq-Av-bdf2}
\frac1\tau \sum_{j=0}^k \delta_j v_{n-j} + A(t_n) v_n 
= f_n,\quad k\le n\le N ,
\end{equation}
with given starting values $v_0,\dotsc,v_{k-1}\in D$, satisfies the following stability property
\begin{equation}\label{BDF_MaxReg2}
\begin{aligned}
&\frac{1}{\tau}\big\|(v_n-v_{n-1} )_{n=k}^N\big\|_{ L^p(X)} 
+ \big\|(v_n )_{n=k}^N\big\|_{ L^p(D)} \\
&\le C\Big(
\big\|(f_n)_{n=k}^N\big\|_{ L^p(X)} 
+\frac{1}{\tau}\big\|(v_i)_{i=0}^{k-1}\big\|_{ L^p(X)} 
+ \big\|(v_i )_{i=0}^{k-1}\big\|_{ L^p(D)} \Big) ,
\end{aligned}
\end{equation}
i.e., discrete maximal parabolic regularity.
\end{proposition}

\begin{proof} 
For $k\le n\le m\le N$, we rewrite the numerical method \eqref{lin-eq-Av-bdf2} in the form
\begin{equation}\label{Eqvn-j}
\frac1\tau \sum_{j=0}^k \delta_j v_{n-j} + A_m v_n 
= f_n +  (A_m - A_n ) v_n,  
\end{equation}
with $A_j:=A(t_j),$ and shall use a discrete perturbation argument.
First, applying the discrete maximal regularity of the implicit $k$-step BDF method
for autonomous equations,
namely \eqref{BDF_MaxReg}, to \eqref{Eqvn-j}, we obtain the estimate
\begin{equation} \label{MaxLp-vn}
\begin{aligned}
&\frac{1}{\tau}\big\|(v_n-v_{n-1} )_{n=k}^m \big\|_{ L^p(X)} 
+ \big\|(v_n )_{n=k}^m \big\|_{ L^p(D)} \\
&\le C\big\|(f_n)_{n=k}^m \big\|_{ L^p(X)} 
+C\big\|\big ( (A_m - A_n)  v_n \big )_{n=k}^m \big\|_{ L^p(X)} \\
&\quad +C\Big( \frac{1}{\tau}\big\|(v_i)_{i=0}^{k-1}\big\|_{ L^p(X)} 
+ \big\|(v_i )_{i=0}^{k-1}\big\|_{ L^p(D)} \Big). 
\end{aligned}
\end{equation}
We now let $E_{k-1}:=0$ and 
\begin{equation}\label{E1}
E_\ell :=\big\|(v_n )_{n=k}^\ell \big\|_{ L^p(D)}^p =\tau\sum_{n=k}^\ell  \|v_n\|_{D}^p,
\quad \ell=k,\dotsc,N,
\end{equation}
and focus on the second term on the right-hand side of \eqref{MaxLp-vn}. Denoting for 
notational simplicity the operator norm $\|\cdot\|_{{\mathcal L}(D,X)}$ by $\|\cdot\|,$ we first note that
\begin{align*}
&{}\big\|\big ( (A_m - A_n ) v_n \big )_{n=k}^m \big\|_{L^p(X)}^p
=\tau\sum_{n=k}^m  \| (A_m - A_n ) v_n\|_{X}^p\\
&{}\le \tau\sum_{n=k}^m  \|A_m - A_n\|^p\|v_n\|_{D}^p
=\sum_{n=k}^m  \|A_m - A_n\|^p(E_n-E_{n-1}),
\end{align*}
whence
\begin{equation}\label{BV1}
\big\|\big ( (A_m - A_n ) v_n \big )_{n=k}^m \big\|_{L^p(X)}^p
\le \sum_{n=k}^{m-1}\big (  \|A_m - A_n\|^p- \|A_m - A_{n+1}\|^p\big )E_n.
\end{equation}
Now, since $\|A_j\|\le M_1,$ it is easily seen that
\[ \big | \|A_m - A_n\|^p- \|A_m - A_{n+1}\|^p\big |\le c_\star\,
\big | \|A_m - A_n\|- \|A_m - A_{n+1}\|\big |\]
with $c_\star:=p(2M_1)^{p-1};$ therefore,
\begin{equation}\label{BV2}
\big | \|A_m - A_n\|^p- \|A_m - A_{n+1}\|^p\big |\le c_\star\|A_{n+1} - A_n\|,
\end{equation}
and \eqref{BV1} yields
\begin{equation}\label{BV3}
\big\|\big ( (A_m - A_n ) v_n \big )_{n=k}^m \big\|_{L^p(X)}^p
\le c_\star\sum_{n=k}^{m-1}\|A_{n+1} - A_n\|E_n.
\end{equation}
Now, letting
\begin{equation*}
F_m :=
\big\|(f_n)_{n=k}^m \big\|_{ L^p(X)}^p    
 +\Big( \frac{1}{\tau}\big\|(v_i)_{i=0}^{k-1}\big\|_{ L^p(X)}   
+ \big\|(v_i )_{i=0}^{k-1}\big\|_{ L^p(D)} \Big)^p  ,
\end{equation*}  
considering the $p^{\rm th}$ power of \eqref{MaxLp-vn}, 
and using \eqref{BV3}, we have 
\[E_m\le C c_\star\sum_{n=k}^{m-1}\|A_{n+1} - A_n\|E_n+CF_m,\]
i.e.,
\begin{equation}\label{BV4} 
E_m\le C\sum_{n=k}^{m-1} a_nE_n +CF_m , \quad m=k,\dotsc,N,
\end{equation} 
with $a_n:=c_\star\|A_{n+1} - A_n\|.$ 
In view of the bounded variation condition \eqref{bv-t}, the sum $\sum_{n=k}^N a_n$
is uniformly bounded by a constant independent of the time step $\tau;$
therefore, a discrete Gronwall-type argument 
applied to \eqref{BV4} yields 
\begin{equation} 
E_N \le  CF_N .
\end{equation} 
In other words, we established the estimate
\begin{equation}  
\begin{aligned}
&{} \big\|(v_n )_{n=k}^N \big\|_{ L^p(D)} \\
&{}\le  C\Big(
\big\|(f_n)_{n=k}^N \big\|_{ L^p(X)}    + \frac{1}{\tau}\big\|(v_i)_{i=0}^{k-1}\big\|_{ L^p(X)} 
+ \big\|(v_i )_{i=0}^{k-1}\big\|_{ L^p(D)} \Big) .
\end{aligned}
\end{equation} 
Then, from \eqref{lin-eq-Av-bdf2} we furthermore obtain 
\begin{equation}
\big\|\big(\frac1\tau \sum_{j=0}^k \delta_j v_{n-j}\big)_{n=k}^N\big\|_{L^p(D)}
\le C\big (\|(f_n)_{n=k}^N\|_{L^p(D)}+ \|(v_n)_{n=k}^N\|_{L^p(D)}\big ) .
\end{equation}
The last two estimates imply \eqref{BDF_MaxReg2}. 
\end{proof}

\section{Angle of analyticity}  \label{Sec:Angle}

Assuming \eqref{en:A1} and \eqref{en:A3}, 
we show here that the angle $\theta$ 
of analyticity of the semigroups generated by $-A_H(s), s\in [0,T],$
exceeds $\arcsin(1/\lambda).$

\begin{lemma}\label{LemmaAngle}
Under assumptions \eqref{en:A1} and \eqref{en:A3}, 
we have $\theta\ge \arcsin(1/\lambda)$. 

\end{lemma}
\begin{proof}
It is known that $-A_H(s)$ generates an analytic semigroup 
in the sector $\varSigma_{\vartheta},$ if and only if the following two conditions hold 
(see, e.g., \cite[Theorem 3.7.11]{ABHN}): 
\begin{enumerate}[label=(\roman*),ref=\roman*]\itemsep=2pt
\item  $z+A_H(s)$ is invertible for $z\in  \varSigma_{\vartheta}$; \label{en:i} 
\item $z(z+A_H(s))^{-1}$ is uniformly bounded on $H$  for $z\in  \varSigma_{\varphi+\pi/2}$, 
for $\varphi\in(0,\vartheta)$. \label{en:ii} 
\end{enumerate}

The analyticity of the semigroup implies that  $z+A_H(s)$ is invertible for $\Real (z)\ge 0$, and 
the compact imbedding $D_H\hookrightarrow\hookrightarrow H$ implies that $z+A_H(s)$ is a 
Fredholm operator of index zero. Hence, in order to verify \eqref{en:i}, we only need to prove 
the injectivity of the operator  $z+A_H(s)$. In fact, if $z\in  \varSigma_{\varphi+\pi/2}$ with 
$\varphi=\arcsin(1/\lambda)$, then $w\in D_H$ and $(z+A_H(s))w=0$ imply 
\[z\|w\|_H^2+(A(s)w,w)=(zw,w)+(A_H(s)w,w)=0;\]
thus, taking real parts, we have 
\[\Real (z)\|w\|_H^2 +\Real (A(s)w,w)=0 .\]
By using \eqref{bounded-a} of assumption \eqref{en:A3}, from the last two relations we see that 
\[|z|\|w\|_H^2=|(A(s)w,w)|\le \lambda \Real (A(s)w,w)=-\lambda \Real (z)\|w\|_H^2,\]
whence
\[(\lambda \Real (z)+|z|)\|w\|_H^2\le 0 .\]
Hence, since $\lambda \Real(z)+|z|>0$ for $z\in\varSigma_{\varphi+\pi/2}$, it follows that 
$\|w\|_H^2=0$. This shows the injectivity of the map  $z+A_H(s):D_H\rightarrow H$, 
which implies invertibility of this Fredholm operator.  
This proves \eqref{en:i} for $\vartheta=\arcsin(1/\lambda)$. 

To verify \eqref{en:ii}, we assume that $z\in  \varSigma_{\varphi+\pi/2}$ 
with $\varphi<\arcsin(1/\lambda)$, and 
$z(z+A_H(s))^{-1}v=w$. Then 
$(z+A_H(s))w=zv$.
Taking in this relation the inner product with $w$, we get 
\[z\|w\|_H^2+(A(s)w,w)=(zw,w)+(A_H(s)w,w)=(zv,w),\]
whence,  taking real parts, 
\[\Real (z)\|w\|_H^2 + \Real (A(s)w,w)=\Real (zv,w) .\]
In view of the last two relations, we have
\begin{align*}
|z|\|w\|_H^2 
&\le |(zv,w)| + |(A(s)w,w)| \\
&\le |z|\|v\|_H\|w\|_H + \lambda \Real (A(s)w,w)
\qquad\text{(in view of \eqref{en:A3})} \\
&= |z|\|v\|_H\|w\|_H + \lambda \Real (zv,w) -\lambda \Real (z)\|w\|_H^2\\
&\le (1+\lambda)|z|\|v\|_H\|w\|_H  -\lambda \Real (z)\|w\|_H^2 , 
\end{align*}
which yields 
\[(\Real (z)/|z|+1/\lambda )\|w\|_H \le (1/\lambda+1)\|v\|_H .\]
Since $\Real (z)/|z|+1/\lambda\ge -\sin\varphi +1/\lambda>0$ 
for $z\in  \varSigma_{\varphi+\pi/2}$ with $\varphi< \arcsin(1/\lambda)$, it follows that 
\[\|z(z+A_H(s))^{-1}v\|_H=\|w\|_H 
\le \frac{1/\lambda+1}{1/\lambda -\sin\varphi}\|v\|_H ,
\quad\forall\, z\in\varSigma_\varphi .\]
Since this estimate is valid for arbitrary $\varphi<\arcsin (1/\lambda)$, it follows that \eqref{en:ii}
is valid for $\vartheta= \arcsin(1/\lambda)$. 
The proof is complete. 
\end{proof}

\section{Stability}\label{Se:stab}
In this section we prove local stability of the implicit--explicit BDF schemes \eqref{abg}.
We shall combine this stability result with the easily established consistency
of the schemes to derive optimal order error estimates in section \ref{Se:cons}.

Besides the approximations $u_n\in D, n=0,\dotsc,N,$ satisfying \eqref{abg},
we consider the nodal values $u^\star_m:=u(t_m)$ of the solution $u$ of the initial value 
problem \eqref{ivp}, which satisfy the perturbed equation
\begin{equation}
\label{ivp2-bdf}
\frac{1}{\tau}\sum\limits^{k}_{i=0}\delta_i u_{n-i}^\star + A(t_n) u_n^\star
=\sum \limits^{k-1}_{i=0}\gamma_iB(t_{n-i-1},u^\star_{n-i-1})
+d_n,  \quad k\le n \le N.\\
\end{equation}
We assume for the time being, and shall verify in the next section, 
that the consistency error $(d_n)$ is  bounded by
\begin{equation}
\label{est-delta-bdf}
 \big\|(d_n)_{n=k}^N\big\|_{ L^p(X)} \le \delta 
\end{equation}
and also that the errors of the starting values are bounded by
\begin{equation}
\label{est-err-start}
\frac1\tau\,   \big\|(u_i - u_i^\star)_{i=0}^{k-1}\big\|_{ L^p(X)}
+\big\|( u_i - u_i^\star)_{i=0}^{k-1}\big\|_{ L^p(D)}  \le \delta,
\end{equation}
with $\delta$ a sufficiently small constant. 
We then have the following stability results for the BDF solutions. 

\begin{proposition}[Stability of the implicit--explicit BDF schemes \eqref{abg}]
\label{prop:stability} 
Consider time discretization of the initial value problem \eqref{ivp}
by the  implicit--explicit $k$-step BDF method \eqref{abg}--\eqref{BDF1}, 
with $1\le k\le 6$ and starting values $u_0,\dots,u_{k-1}\in D$, and assume that
the stability condition \eqref{intr-stab-abg} is satisfied.
Under the assumptions \eqref{en:A1}--\eqref{en:A5} and \eqref{est-delta-bdf}--\eqref{est-err-start}, 
there exist positive constants $\tilde\lambda_0$ and $\delta_0$ such that, for 
$\tilde\lambda\le\tilde\lambda_0$ and $\delta\le\delta_0$, the errors $e_n=u_n-u_n^\star$ 
between the solutions of \eqref{abg} and \eqref{ivp2-bdf} are bounded by
\begin{align}
&\frac{1}{\tau}\big\|(e_n-e_{n-1} )_{n=k}^N\big\|_{ L^p(X)} +  \big\|(e_n )_{n=k}^N\big\|_{ L^p(D)}
 \le C\delta , \label{conv1}\\
& \big\|(e_n )_{n=k}^N\big\|_{ L^\infty(W)} \le C\delta, \label{conv2}
\end{align}
with a constant $C$ depending on 
$\| (u^\star_n)_{n=0}^N \|_{L^\infty(W)}$, $\| (u^\star_n)_{n=0}^N \|_{L^p(D)}$, and $T$, 
but independent of $\delta$ and $\tau$. 
\end{proposition}

\begin{proof}
Subtracting \eqref{ivp2-bdf} from \eqref{abg}, we obtain the following error equation,
for the errors $e_n:=u_n-u^\star_n,$ 
\begin{equation}
\label{er-eq-bdf}
\frac{1}{\tau}\sum\limits^{k}_{i=0}\delta_i e_{n-i} + A(t_n)e_n
=\sum\limits^{k-1}_{i=0}\gamma_ib_{n-i-1}-d_n,
\quad n=k,\dotsc,N,
\end{equation}
with the abbreviation $b_\ell:=B(t_\ell,u_\ell)-B(t_\ell,u^\star_\ell), \ell=0,\dotsc,N-1.$

We let $r\in(0,r_0]$ be a small number to be determined later, and let $M\le N$ be maximal such that 
\begin{equation} \label{M-RB}
\| (e_n)_{n=0}^{M-1} \|_{L^\infty(W)} \le r  .
\end{equation}
%
Then $u_n\in T_{u,r}^D, n=0,\dotsc,M-1,$ 
and assumption \eqref{en:A4} implies 
\begin{equation*}
\begin{aligned}
\| b_\ell \|_X  
&\le \tilde\lambda\|e_\ell\|_D+C_B(\|u_\ell\|_D+\|u_\ell^\star\|_D)\|e_\ell\|_W \\
&\le \tilde\lambda\|e_\ell\|_D+C_B(\|e_\ell\|_D+2\|u_\ell^\star\|_D)\|e_\ell\|_W \\
&= (\tilde\lambda+C_B\|e_\ell\|_W)\|e_\ell\|_D+2C_B\|u_\ell^\star\|_D \|e_\ell\|_W  
\end{aligned}
\end{equation*}
and thus
\begin{equation}\label{M-RBb}
\| b_\ell \|_X \le (\tilde\lambda+C_Br)\|e_\ell\|_D+2C_B\|u_\ell^\star\|_D \|e_\ell\|_W,
\quad \ell=0,\dotsc,M-1 .
\end{equation}

Let us denote by $J_m$ the quantity
\begin{equation}\label{M-RBJ}
J_m:=\frac{1}{\tau}\big\|(e_n-e_{n-1} )_{n=k}^m\big\|_{ L^p(X)} +  \big\|(e_n )_{n=k}^m\big\|_{ L^p(D)}, 
\end{equation}
which we want to estimate; cf.\ the left-hand side of \eqref{conv1}.
First, applying Proposition \ref{Prop:MaxLp} to \eqref{er-eq-bdf},  
we obtain, for all $m\le M$,
\begin{equation*} 
\begin{aligned}
J_m&\le 
 C\,  \bigg\|\Big(\sum\limits^{k-1}_{i=0}\gamma_ib_{n-i-1}-d_n\Big)_{n=k}^m\bigg\|_{ L^p(X)} 
 +\frac{C}{\tau}\,   \big\|(e_i)_{i=0}^{k-1}\big\|_{ L^p(X)}
+C\big\|(e_i)_{i=0}^{k-1}\big\|_{ L^p(D)} \\
& \le  C\,  \big\|(b_{n})_{n=0}^{m-1}\big\|_{ L^p(X)} 
 +C\big\|(d_n)_{n=k}^m\big\|_{ L^p(X)}
 +\frac{C}{\tau}\,   \big\|(e_i)_{i=0}^{k-1}\big\|_{ L^p(X)}
+C\big\|(e_i)_{i=0}^{k-1}\big\|_{ L^p(D)}  \\[5pt]
&  \le C\,  \big\|(b_{n})_{n=0}^{m-1}\big\|_{ L^p(X)} +C\delta ,
\end{aligned}
\end{equation*}
where we used \eqref{est-delta-bdf} and \eqref{est-err-start};  
therefore, in view of \eqref{M-RBb}, 
\[J_m\le C(\tilde\lambda+r)\|(e_n)_{n=0}^{m-1}\|_{L^p(D)}  +C\|(e_n)_{n=0}^{m-1}\|_{L^p(W)} +C\delta.\]
Since \eqref{en:A5} implies 
\[\|(e_n)_{n=0}^{m-1}\|_{L^p(W)} \le \varepsilon\|(e_n)_{n=0}^{m-1}\|_{L^p(D)}
+C_\varepsilon\|(e_n)_{n=0}^{m-1}\|_{L^p(X)} ,\]
the last two estimates yield 
\begin{equation}\label{MaxLp-new1} 
\begin{aligned}
J_m
&\le C(\tilde\lambda+r)\|(e_n)_{n=0}^{m-1}\|_{L^p(D)}  \\
&\quad 
+\varepsilon\|(e_n)_{n=0}^{m-1}\|_{L^p(D)}
+C_\varepsilon\|(e_n)_{n=0}^{m-1}\|_{L^p(X)}
+C\delta .
\end{aligned}
\end{equation}
Now, with a sufficiently small positive constant $\tilde\lambda_0,$ 
for $\tilde\lambda\le\tilde\lambda_0$ and suitably small 
$r$ and $\varepsilon,$ we have $C(\tilde\lambda+r)+\varepsilon\le 1/2$ ($r$ is a 
fixed constant from now on). 
As a consequence, \eqref{MaxLp-new1}  yields
\begin{equation}\label{MaxLp-new2}  
J_m\le C\|(e_n)_{n=0}^{m-1}\|_{L^p(X)} +C\delta .
\end{equation}
Combining this estimate with the upper bound \eqref{est-err-start} of the starting errors, 
and setting $e_{-1}=0$, we obtain 
\begin{equation} \label{MaxLpent}
\frac{1}{\tau}\big\|(e_n-e_{n-1} )_{n=0}^m\big\|_{ L^p(X)} 
+  \big\|(e_n )_{n=0}^m\big\|_{ L^p(D)}\le C\|(e_n)_{n=0}^{m-1}\|_{L^p(X)} +C\delta .
\end{equation}
Now, with $p'$ such that $\frac 1p +\frac 1{p'}=1,$  we have
\begin{equation*} \label{en_LinftyX}
\begin{aligned}
\|(e_n)_{n=0}^m\|_{L^\infty(X)}
&\le \|e_0\|_{X}
+\sum_{n=1}^m\big\|e_n-e_{n-1} \big\|_{ X} \\
&= \|e_0\|_{X}
+\frac{1}{\tau}\big\|(e_n-e_{n-1})_{n=1}^m \big\|_{L^1(X)} \\
&\le \tau^{1/p'}\delta
+\frac{T^{\frac{1}{p'}}}{\tau}\big\|(e_n-e_{n-1})_{n=1}^m \big\|_{L^p(X)} \\
&\le \tau^{1/p'}\delta+ C\|(e_n)_{n=0}^{m-1}\|_{L^p(X)}
+C\delta \\
&\le \varepsilon\|(e_n)_{n=0}^{m-1}\|_{L^\infty(X)}
+C_\varepsilon\|(e_n)_{n=0}^{m-1}\|_{L^1(X)}
+C\delta;
\end{aligned}
\end{equation*}
%
in the last step of this estimate we used the following inequality \cite[Theorem 2.11]{Adams}:
\begin{equation*} 
\begin{aligned}
\|(e_n)_{n=0}^{m-1}\|_{L^p(X)}
&\le 
\|(e_n)_{n=0}^{m-1}\|_{L^\infty(X)}^{1-\frac{1}{p}}
\|(e_n)_{n=0}^{m-1}\|_{L^1(X)}^{\frac{1}{p}} \\
&\le  \varepsilon\|(e_n)_{n=0}^{m-1}\|_{L^\infty(X)}
+C_\varepsilon\|(e_n)_{n=0}^{m-1}\|_{L^1(X)} .
\end{aligned}
\end{equation*}
Therefore, 
\begin{equation} 
\begin{aligned}
\|(e_n)_{n=0}^m\|_{L^\infty(X)}
&\le C\|(e_n)_{n=0}^{m-1}\|_{L^1(X)}
+C\delta ,
\end{aligned}
\end{equation}
which holds for all $0\le m\le M$. 
Then Gronwall's inequality implies $\|(e_n)_{n=0}^M\|_{L^\infty(X)}
\le C\delta .$
Substituting this estimate into \eqref{MaxLpent}, we obtain 
\begin{equation}  \label{FnlMaxLp}
\begin{aligned}
&\frac{1}{\tau}\big\|(e_n-e_{n-1} )_{n=0}^M\big\|_{ L^p(X)} 
+  \big\|(e_n )_{n=0}^M\big\|_{ L^p(D)}  
\le C\delta .
 \end{aligned}
\end{equation}
Let $\widetilde e\in W^{1,p}(-\tau,M\tau;X)\cap L^p(-\tau,M\tau;D)$ denote the piecewise linear 
interpolant  of $e_n$, $n=-1,0,1,\dotsc,M$, at the points $t_n$. Then we have 
\[\|\widetilde e\,'\|_{ L^p(-\tau,M\tau;X)} 
+ \|\widetilde e\|_{ L^p(-\tau,M\tau;D)} 
\le C\bigg(\frac{1}{\tau}\big\|(e_n-e_{n-1} )_{n=0}^M\big\|_{ L^p(X)} 
+  \big\|(e_n )_{n=0}^M\big\|_{ L^p(D)} \bigg),\]
whence, in view of \eqref{FnlMaxLp},
\begin{equation} \label{FnlMaxLp3}  
\|\widetilde e\,'\|_{ L^p(-\tau,M\tau;X)} 
+ \|\widetilde e\|_{ L^p(-\tau,M\tau;D)}  \le C\delta.
\end{equation}
Now, according to assumption \eqref{en:A5}, we have
\begin{equation*}
\big\|(e_n )_{n=0}^M\big\|_{ L^\infty(W)}
=\|\widetilde e\|_{ L^\infty(-\tau,M\tau;W)} \\
\le C(\|\widetilde e\,'\|_{ L^p(-\tau,M\tau;X)} 
+  \|\widetilde e\|_{ L^p(-\tau,M\tau;D)}),
\end{equation*}
and, using \eqref{FnlMaxLp3}, we obtain
\begin{equation}   \label{FnlMaxLp2}
\big\|(e_n )_{n=0}^M\big\|_{ L^\infty(W)}\le C\delta .
\end{equation}
Estimates \eqref{FnlMaxLp} and \eqref{FnlMaxLp2} imply the existence of a positive 
constant $\delta_0$ such that $\|(e_n)_{n=0}^M\|_{L^\infty(W)} \le r $
for $\delta\le\delta_0$,   contradicting the maximality of $M$ unless $M=N$. 
Hence, \eqref{M-RB}, \eqref{FnlMaxLp} and \eqref{FnlMaxLp2} are valid for $M=N$. 
The proof is complete.
\end{proof}

\begin{remark}[On the stability condition \eqref{intr-stab-abg}]\label{Re:stab-condition}
The stability condition \eqref{intr-stab-abg}  is void for $k=1,2,$ and takes for the 
implicit--explicit $k$-step BDF method \eqref{abg}, $k=3,\dotsc,6,$ the form 
$\lambda <\lambda_k$ with
\begin{equation}
\label{lambdak}
\lambda_3=14.45087,\quad \lambda_4=3.49040,\quad \lambda_5=1.62892979,
\quad \lambda_6=1.050513.
\end{equation}
\end{remark}

\section{Error estimates}\label{Se:cons}
In this section we present our main result, maximum norm optimal order error estimates
for the implicit--explicit BDF schemes \eqref{abg}.

\begin{proposition}[Optimal order error estimates]\label{Prop:Error}
Assume \eqref{en:A1}--\eqref{en:A5}, with stability constant satisfying \eqref{intr-stab-abg},
and that the starting approximations are such that
\begin{equation}
\label{est-err-start-conv}
\frac1\tau\,   \big\|(u_i - u_i^\star)_{i=0}^{k-1}\big\|_{ L^p(X)}
+\big\|( u_i - u_i^\star)_{i=0}^{k-1}\big\|_{ L^p(D)}  \le C\tau^k.
\end{equation}
If the solution $u$ of \eqref{ivp} is sufficiently regular, 
$u\in C^{k+1}([0,T];X)\cap C^k([0,T];D)$,  then there exist positive constants $\tilde\lambda_0$ 
and $\tau_0$ such that, for $\tilde\lambda\le\tilde\lambda_0$ and $\tau\le \tau_0$, the 
errors $e_n=u_n-u_n^\star$ between the approximate solutions $u_n$ of \eqref{abg} and 
the nodal values $u_n^\star$ of the solution $u$ of \eqref{ivp} are bounded by
\begin{align}
&\frac{1}{\tau}\big\|(e_n-e_{n-1} )_{n=k}^N\big\|_{ L^p(X)} +  \big\|(e_n )_{n=k}^N\big\|_{ L^p(D)}
 \le C\tau^k , \label{er-est1}\\
& \big\|(e_n )_{n=k}^N\big\|_{ L^\infty(W)} \le C\tau^k , \label{er-est2}
\end{align}
with a constant $C$ independent of $\tau$.
\end{proposition}

\begin{proof}
With Proposition \ref{prop:stability}  on the stability of the BDF solutions, 
we only need to establish estimates for the consistency errors $d_n$.

The order of the $k-$step methods $(\delta,\beta)$ and $(\delta,\gamma)$ is $k,$ i.e.,
\begin{equation}
\label{order}
\sum_{i=0}^k(k-i)^\ell \delta_i=\ell k^{\ell-1}=
\ell \sum_{i=0}^{k-1}(k-i-1)^{\ell-1} \gamma_i,\quad \ell=0,1,\dotsc,k.  
\end{equation}
The consistency error $d_n$ of the scheme \eqref{abg} for the solution $u$ of \eqref{ivp}, 
i.e., the amount by which the exact solution misses satisfying the implicit--explicit BDF scheme
\eqref{abg}, is given by 
\begin{equation}
\label{cons1}
\tau d_n=\sum\limits^k_{i=0}\delta_iu(t_{n-i}) + \tau A(t_n)u(t_n)
-  \tau \sum\limits^{k-1}_{i=0} \gamma_i B\big (t_{n-i-1},u(t_{n-i-1})\big ),
\end{equation}
$n=k,\dotsc,N;$ cf.\ \eqref{ivp2-bdf}. 
Letting
\[\left\{
\begin{aligned}
&d_{n,1}:=\sum\limits^k_{i=0}\delta_iu(t_{n-i}) - \tau u'(t_n), \\
&d_{n,2}:=\tau B\big (t_n,u(t_n)\big )- 
\tau\sum\limits^{k-1}_{i=0} \gamma_i B\big (t_{n-i-1},u(t_{n-i-1})\big ),
\end{aligned}
\right.\]
and using the differential equation in \eqref{ivp}, we infer that $\tau d_n=d_{n,1}+d_{n,2}.$
Now, by Taylor expanding about $t_{n-k}$ and using the
order conditions of the implicit $(\delta,\beta)$-scheme,
i.e., the first equality in \eqref{order}, and the second equality 
in \eqref{order}, respectively, we obtain
\begin{equation*}
\begin{aligned}
 d_{n,1}&{}=\frac 1{k!}\Bigg [ \sum\limits^k_{i=0} \delta_i\!
\int_{t_{n-k}}^{t_{n-i}}(t_{n-i}-s)^ku^{(k+1)}(s)\, ds
-k\tau \int_{t_{n-k}}^{t_n}(t_n-s)^{k-1}u^{(k+1)}(s)\, ds\Bigg ]\!,\\
 d_{n,2}&{}=\frac \tau {(k-1)!}\Bigg [ \int_{t_{n-k}}^{t_n}(t_n-s)^{k-1}\widetilde B^{(k)}(s)\, ds
-\sum\limits^k_{i=0}\gamma_i \!\int_{t_{n-k}}^{t_{n-i-1}}(t_{n-i-1}-s)^{k-1}
\widetilde B^{(k)}(s)\, ds\Bigg ]\!,
\end{aligned}
\end{equation*}
with $\widetilde B(t):=B(t,u(t)), t\in [0,T].$ 
Thus, under the regularity condition
\[u\in C^{k+1}([0,T];X)\cap C^k([0,T];D) ,\]
we obtain the desired optimal order consistency estimate  
%
\begin{equation}
\label{cons6}
\max_{k\le n\le N}\|d_n\|_X \le C\tau^{k} .
\end{equation}

Now, in view of the consistency estimate \eqref{cons6} and our assumption 
\eqref{est-err-start-conv} on the starting approximations, conditions 
\eqref{est-delta-bdf} and  \eqref{est-err-start} are valid with $\delta=C\tau^{k} .$
Therefore, for sufficiently small time step $\tau,$ the desired error
estimates \eqref{er-est1} and \eqref{er-est2} follow immediately from 
the corresponding estimates \eqref{conv1} and \eqref{conv2}, respectively.
\end{proof}

\begin{remark}[On the accuracy requirement for the starting approximations]\label{Re:start-accur}
The accuracy requirement \eqref{est-err-start-conv}  for the starting approximations
$u_0,\dotsc,u_{k-1}$ can be equivalently written in the form 
%
\begin{equation}
\label{start-accur1}
\max_{0\le i\le k-1}\|u_i - u_i^\star\|_X\le C\tau^{k+(1-\frac 1p)},\quad
\max_{0\le i\le k-1}\|u_i - u_i^\star\|_D\le C\tau^{k-\frac 1p}.
\end{equation}
The larger $p$ is, the stronger is the accuracy requirement \eqref{start-accur1} on the
starting approximations. 
Suitable choices of $p$ depend on the concrete application; for 
instance, for the problems in Examples \ref{Exmp1}--\ref{Exmp2}, \ref{Exmp3}, and \ref{Exmp4}, 
respectively, the conditions $2/p+d/q<1,1/p+d/q<1,$ and $4/p+d/q<2,$ respectively, are required; 
see Propositions \ref{lem:framework}--\ref{lem:framework4} and Section \ref{Sec:ProofProposition}.
\end{remark}

Proposition \ref{Prop:Error} applies directly to Examples \ref{Exmp1}--\ref{Exmp4};
for instance, in the cases of Examples \ref{Exmp1} and \ref{Exmp2}, we have:

\begin{corollary}[Application to Examples \ref{Exmp1} and \ref{Exmp2}]\label{Coroll:Error}
Assume that the stability constant satisfies \eqref{intr-stab-abg}.
Then, in the case of Example \ref{Exmp1}, 
if the solution $u$ of \eqref{eq:Exmp1} is sufficiently regular, 
\[u\in C^{k+1}\big ([0,T];W^{-1,q}(\varOmega)\big ) \cap C^k\big ([0,T];W^{1,q}_0(\varOmega)\big ) , \]
the starting approximations are such that
\begin{equation}
\frac1\tau\,   \big\|(u_i - u_i^\star)_{i=0}^{k-1}\big\|_{ L^p (W^{-1,q}(\varOmega) )}
+\big\|( u_i - u_i^\star)_{i=0}^{k-1}\big\|_{ L^p (W^{1,q}(\varOmega) )}  \le C\tau^k,
\end{equation}
with $q\in(d,\infty)\cap [2,\infty)$ and $p$ such that $2/p+d/q<1,$
and the coefficients $a$ and $b$ satisfy \eqref{intr-stab-abg3},
then, for sufficiently small time step
 $\tau$, the errors $e_n=u_n-u_n^\star$ between the approximate solutions $u_n$ of \eqref{abg} and 
the nodal values $u_n^\star$ of the solution $u$ of \eqref{eq:Exmp1} are bounded by
\begin{align}
&\frac{1}{\tau}\big\|(e_n-e_{n-1} )_{n=k}^N\big\|_{ L^p(W^{-1,q}(\varOmega))} 
+  \big\|(e_n )_{n=k}^N\big\|_{ L^p(W^{1,q}(\varOmega))}
 \le C\tau^k , \label{er-est1a}\\
&  \max_{k\le n\le N}   \|e_n \|_{L^\infty(\varOmega)} \le C\tau^k ,\label{er-est2a}  
\end{align}
with a constant $C$ independent of $\tau$.

Analogously, in the case of Example \ref{Exmp2},  
if the solution $u$ of \eqref{eq:Exmp2} is sufficiently regular, 
\[
u\in C^{k+1}\big ([0,T];L^q(\varOmega)\big )\cap C^k\big ([0,T]; 
W^{2,q}(\varOmega)\cap W_0^{1,q}(\varOmega)\big ),
\]
the starting approximations are such that
\begin{equation}
\frac1\tau\,   \big\|(u_i - u_i^\star)_{i=0}^{k-1}\big\|_{ L^p(L^q(\varOmega))}
+\big\|( u_i - u_i^\star)_{i=0}^{k-1}\big\|_{ L^p(W^{2,q}(\varOmega))}  \le C\tau^k,
\end{equation}
with $q>d$ and $p$ such that $2/p+d/q<1,$ and the coefficients $a$ and $b$ 
satisfy \eqref{intr-stab-abg3}, 
then for a sufficiently small constant $\tilde\lambda$ in \eqref{tildelambda-Exam3}
and for sufficiently small time step  $\tau$, the 
errors $e_n=u_n-u_n^\star$ between the approximate solutions $u_n$ of \eqref{abg} and 
the nodal values $u_n^\star$ of the solution $u$ of \eqref{eq:Exmp2}  are bounded by
\begin{align}
&\frac{1}{\tau}\big\|(e_n-e_{n-1} )_{n=k}^N\big\|_{ L^p(L^q(\varOmega))} 
+  \big\|(e_n )_{n=k}^N\big\|_{ L^p(W^{2,q}(\varOmega))}
 \le C\tau^k , \label{er-est1b}\\
&  \max_{k\le n\le N}   \|e_n \|_{W^{1,\infty}(\varOmega)} \le C\tau^k ,\label{er-est2b}  
\end{align}
with a constant $C$ independent of $\tau$.
\end{corollary}

%
%

\section{Proofs of Propositions \ref{lem:framework}, \ref{lem:framework3}, and \ref{lem:framework4}}
\label{Sec:ProofProposition}

In this section we show that our abstract framework is applicable in the cases
of Examples \ref{Exmp1}--\ref{Exmp4}.

\subsection{Proof of Proposition \ref{lem:framework}}
In this subsection we show that our abstract framework is applicable in the cases
of Examples \ref{Exmp1} and \ref{Exmp2}; more precisely, we verify our abstract
conditions \eqref{en:A1}--\eqref{en:A5} for the initial and boundary value problems
\eqref{eq:Exmp1} and \eqref{eq:Exmp2} in the spaces given in Proposition \ref{lem:framework}.

It follows easily from the smoothness of the diffusion coefficients  and the positivity of 
the coefficient $a(x,t)$ that assumption \eqref{en:A3} is satisfied; the smallest 
possible value of the stability constant $\lambda$ in \eqref{bounded-a} is
as on the left-hand side of \eqref{intr-stab-abg2}.
Assumption \eqref{en:A4} is satisfied, if the functions $f$ and $\bm{g}$ are locally Lipschitz 
continuous with respect to the arguments $u$ and $\nabla u$. 
The proof of \eqref{en:A5}, under the condition $2/p+d/q<1,$ is given in \cite[Section 9]{ALL}.
Thus, it remains to verify assumptions \eqref{en:A1} and \eqref{en:A2}. 

Existence and uniqueness of a solution $v\in H^1(\R_+;V')\cap L^2(\R_+;V)
\hookrightarrow C([0,\infty);H)$ for \eqref{HomoEq} 
are simple consequences of  the positivity of $a(x,t)$ and the boundedness of 
$a(x,t)$ and $b(x,t)$. 
It is easy to see that the operator $-A_H(s):D_H\rightarrow H,$ 
\[-A_H(s)v=-\nabla\cdot\big((a(x,s)+\i\, b(x,s)) \nabla v \big),\]
is densely defined, closed and invertible.  Following Section \ref{Sec:Angle}, 
one can prove that the operator $-A_H(s)$ generates a bounded analytic semigroup 
in the sector $\varSigma_{\theta_s}$, with 
\[\theta_s\ge  \inf_{x\in\varOmega} \arctan\frac{a(x,s)}{|b(x,s)|} 
\ge
\inf_{(x,t)\in\varOmega\times(0,T)} \arctan\frac{a(x,t)}{|b(x,t)|}  
=\arcsin \frac{1}{\lambda} .\] 
In fact, we have the sharper result 
\begin{equation}\label{AngleE}
\theta_s=  \inf_{x\in\varOmega} \arctan\frac{a(x,s)}{|b(x,s)|}  ,
\end{equation}
which immediately implies 
%
$\theta:=  \inf_{s\in [0,T]} \theta_s =\arcsin \frac{1}{\lambda} $.
%
To prove \eqref{AngleE}, simply note that otherwise we could choose  $\varphi$ 
satisfying
\begin{equation}\label{AngleL}
\theta_s>\varphi>\arctan\frac{a(x_0,s)}{|b(x_0,s)|} 
\quad\mbox{for some \,$x_0\in\varOmega$\, such that \,$b(x_0,s)\ne 0$} .
\end{equation}
Without loss of generality, we can assume that $b(x_0,s)> 0.$ Then the operator
\[-e^{\i\varphi} A_H(s)=\nabla\cdot\big(\big(
a(x,s)\cos\varphi-b(x,s)\sin\varphi+\i\big[a(x,s)\sin\varphi+b(x,s)\cos\varphi\big]\big)\nabla \big)\]
would generate a bounded analytic semigroup. 
But this is impossible, because \eqref{AngleL} 
implies that 
$a(x,s)\cos \varphi-b(x,s)\sin\varphi<0$ at some point $x\in\varOmega$, 
which means that the operator $-e^{\i\varphi}A_H(s)$ would have 
eigenvalues with positive real parts, 
thus the resolvent operator $z+e^{\i\varphi}A_H(s)$ would not be invertible 
for some $z$ on the right half-plane, which contradicts the bounded analyticity of the 
semigroup generated by $-e^{\i\varphi}A_H(s)$ (see, e.g., \cite[Theorem 3.7.11]{ABHN}). 
This proves \eqref{en:A1} for both Problems \eqref{eq:Exmp1} and \eqref{eq:Exmp2}, with \eqref{AngleE}. 

Next, we verify \eqref{en:A2} for Problem \eqref{eq:Exmp2}. 
In this case, $X=L^q(\varOmega)$. 
With the positivity of $a(x,t)$ and the H\"older continuity of  $a(x,t)$ and $b(x,t)$, 
Auscher, McIntosh and Tchamitchian \cite[Theorem 4.19]{AIT98}  have proved that the kernel 
of the semigroup $E_s^H(t)$ has a Gaussian upper bound: 
\begin{equation} 
|G_s(t,x,y)|\leq \frac{C_0}{t^{d/2}}e^{-\frac{|x-y|^2}{C_0t}} ,\quad
\forall\, t>0,\,\,\forall\, x,y\in\varOmega , 
\end{equation}
where the constant $C_0$ depends only on  the lower bound of $a(x,t),$ on the upper bounds of 
$|a(x,t)+\i\, b(x,t)|,$  and on the H\"older norms of $a(\cdot,t)$ and $b(\cdot,t)$, 
but is independent of $s\in[0,T]$. 
Similarly, when $\varphi\in(0,\theta)$ the semigroup generated by $-e^{\i\varphi}A_H(s)$ 
has also Gaussian upper bound. Hence, \cite[Theorem 8.5]{KW} implies that the semigroup generated 
by $-e^{\i\varphi}A_H(s)$ satisfies the conditions of \cite[Theorem 8.6]{KW}, for all $\varphi\in(0,\theta)$, 
which further implies that $\{E_s^H(t)\}_{t>0}$ extends to an $R$-bounded 
analytic semigroup $\{E_s^X(z)\}_{z\in\varSigma_{\varphi}}$ on $X$ (in view of 
\cite[Remark 8.23]{KW} and \cite[Theorem 4.2]{Weis2}). 
If we denote the generator of  $\{E_s^X(t)\}_{t>0}$ by $-A_X(s)$, then \cite[Theorem 4.2]{Weis2} 
implies that the family of operators $\{z(z+A_X(s))^{-1} : z\in\varSigma_\varphi\}$ is $R$-bounded. 
As a consequence,  \cite[Theorems 4.1--4.2 and Remark 4.3]{KLL} implies that, when 
$\theta>\pi/2-\alpha_k,$ the solution of \eqref{lin-eq-Av-bdf} 
satisfies (with $v_{n}=0$ for $n<0$)
\begin{equation}\label{MaxRegLq}
\begin{aligned}
&\frac{1}{\tau}\Big\|\Big(\sum_{j=0}^k \delta_j v_{n-j}\Big)_{n=0}^N\Big\|_{ L^p(L^q(\varOmega))} 
+ \big\|(v_n )_{n=k}^N\big\|_{ L^p(W^{2,q}(\varOmega))} \\
&\le C\Big(
\big\|(f_n)_{n=k}^N\big\|_{ L^p(L^q(\varOmega))} 
+\frac{1}{\tau}\big\|(v_i)_{i=0}^{k-1}\big\|_{ L^p(L^q(\varOmega))} 
+ \big\|(v_i )_{i=0}^{k-1}\big\|_{ L^p(W^{2,q}(\varOmega))} \Big) .
\end{aligned}
\end{equation}
Then \eqref{en:A2} follows from \eqref{MaxRegLq} and the inequality  
\[\frac{1}{\tau}\big\|\big(v_n-v_{n-1}\big)_{n=0}^N\big\|_{ L^p(L^q(\varOmega))} \le \frac{C}{\tau}
\Big\|\Big(\sum_{j=0}^k \delta_j v_{n-j}\Big)_{n=0}^N\Big\|_{ L^p(L^q(\varOmega))} .\]
To prove the last inequality, let $\dot v_{n}:=\frac{1}{\tau}\sum_{j=0}^k \delta_j v_{n-j}$. 
Recall that $\delta(\zeta)=(1-\zeta)\mu(\zeta)$, where the polynomial $\mu(\zeta)$ of degree 
$k-1$ has no zeros in the closed unit disc; 
therefore
\[\frac1{\mu(\zeta)} = \sum_{n=0}^\infty \chi_n \,\zeta^n, \qquad\text{with}\,\,\,
|\chi_n|\le C\rho^n \,\,\, \text{for some}\,\,\,\rho<1.\]
Then, with $\eta_{\ell}=1$ for $\ell\ge 0$ and $\eta_{\ell}=0$ for $\ell< 0,$
 we have 
\[\frac{v_n-v_{n-1}}\tau 
= \sum_{m=0}^n \dot v_{n-m} \, \chi_m
= \sum_{m=0}^N \eta_{n-m} \dot v_{n-m} \, \chi_m ,\]
%
because both sides 
have the same generating function.  
Therefore,
\begin{align*}
\frac{1}{\tau}\Big\|(v_n-v_{n-1})_{n=0}^N\Big\|_{ L^p(X)} 
&\le \sum_{m=0}^N |\chi_m|\, \|(\eta_{n-m}\dot v_{n-m})_{n=0}^N \|_{ L^p(X)} \\ 
&\le \sum_{m=0}^N |\chi_m|\, \|(\dot v_{n})_{n=0}^N \|_{ L^p(X)} \\ 
&\le C\big\|(\dot v_n )_{n=0}^N\big\|_{ L^p(X)}.
\end{align*}

Finally, we verify \eqref{en:A2} for Problem \eqref{eq:Exmp1}. In this case, we want 
to derive from \eqref{MaxRegLq}  the following estimate:
\begin{equation}\label{MaxRegW-1q}
\begin{aligned}
&\frac{1}{\tau}\Big\|\Big(\sum_{j=0}^k \delta_j v_{n-j}\Big)_{n=k}^N\Big\|_{ L^p(W^{-1,q}(\varOmega))} 
+ \big\|(v_n )_{n=k}^N\big\|_{ L^p(W^{1,q}(\varOmega))}\le \\
& C\Big(
\big\|(f_n)_{n=k}^N\big\|_{ L^p(W^{-1,q}(\varOmega))} 
+\frac{1}{\tau}\big\|(v_i)_{i=0}^{k-1}\big\|_{ L^p(W^{-1,q}(\varOmega))} 
+ \big\|(v_i )_{i=0}^{k-1}\big\|_{ L^p(W^{1,q}(\varOmega))} \Big) .
\end{aligned}
\end{equation}
In fact, \cite[Theorem 1]{AQ02} implies that the Riesz transform $\nabla A_X(s)^{-1/2}$ 
is bounded on $L^q(\varOmega)$ for all $1<q<\infty$. 
Once we have \eqref{MaxRegLq} and the boundedness 
of the Riesz transform, the discrete maximal regularity \eqref{BDF_MaxReg} 
can be proved in the same way as \cite[Proposition 8.7]{ALL}. 
\qed


\subsection{Proof of Proposition \ref{lem:framework3}}

Since the operator $-(-\varDelta)^{1/2}$ is self-adjoint and non-positive definite, 
it generates a bounded analytic semigroup of angle $\pi/2$ on the Hilbert space $H=L^2(\R^d)$, 
with domain $D_H=H^1(\R^d)$. This implies \eqref{en:A1}.

Assumption \eqref{en:A3} is due to the time-independence and self-adjointness of 
the operator $(-\varDelta)^{1/2}$. The self-adjointness also implies $\lambda=1$. 

If $w,v$ are bounded in $W=L^2(\R^d)\cap L^\infty(\R^d)\hookrightarrow L^q(\R^d) $, 
with $\|w\|_{W}+\|v\|_{W}\le K$, then the local Lipschitz continuity of $f$ implies, for $q\ge 2$,  
\begin{align*}
\|f(w)-f(v)\|_{X} 
&=\|f(w)-f(v)\|_{L^2(\R^d)}
+\|f(w)-f(v)\|_{L^q(\R^d)}\\
&\le C_K\big (\|w-v\|_{L^2(\R^d)} +\|w-v\|_{L^q(\R^d)}\big )\\
&\le C_K \Big (\|w-v\|_{L^2(\R^d)} +\|w-v\|_{L^2(\R^d)}^{\frac{2}{q}}
\|w-v\|_{L^\infty(\R^d)}^{1-\frac{2}{q}}\Big )\\
&\le C_K \big (\|w-v\|_{L^2(\R^d)} +\|w-v\|_{L^\infty(\R^d)}\big )
= C_K \|w-v\|_{W} .
\end{align*}
This implies assumption \eqref{en:A4} in the case $q\ge 2$, 
with $\tilde\lambda=0$.

If $q>d$, then \cite[Theorems 5.2 and 5.9]{Adams} implies that, for any positive $\varepsilon$, there 
exists positive $C_\varepsilon$ such that
\[\|v\|_{L^q(\R^d)}+\|v\|_{L^\infty(\R^d)}  
\le \varepsilon \|v\|_{W^{1,q}(\R^d)}
+C_\varepsilon \|v\|_{L^q(\R^d)} 
\quad\forall\, v\in W^{1,q}(\R^d) .\]
For $D=H^1(\R^d)\cap W^{1,q}(\R^d)$ 
and $X=L^2(\R^d)\cap L^q(\R^d)$, 
we have 
\[W^{1,p}(0,T;X)\cap L^p(0,T;D)
\hookrightarrow L^\infty(0,T;X)
\hookrightarrow L^\infty(0,T;L^2(\R^d)),
\quad\text{when $p>1$}, \]
and 
\begin{equation*}
\begin{aligned}
&W^{1,p}(0,T;X)\cap L^p(0,T;D)  
\hookrightarrow 
W^{1,p}(0,T;L^q(\R^d))\cap L^p(0,T;W^{1,q}(\R^d))  \\
&\hookrightarrow 
L^\infty(0,T;(L^q(\R^d),W^{1,q}(\R^d))_{1-1/p,p})
\qquad\quad \text{see \cite[Proposition 1.2.10]{Lunardi95}}\\
&=
L^\infty(0,T;B^{1-1/p;q,p}(\R^d)) 
\quad\,\,\, \text{by the definition of Besov spaces 
\cite[\textsection 7.32]{Adams}} \\
&\hookrightarrow 
L^\infty(0,T;L^{\infty}(\R^d)), 
\qquad\quad\, \text{when $1/p+d/q<1$, see \cite[\textsection 7.34]{Adams}} .
\end{aligned}
\end{equation*} 
The last two imbedding results imply 
\[W^{1,p}(0,T;X)\cap L^p(0,T;D)
\hookrightarrow L^\infty(0,T;L^2(\R^d)\cap L^\infty(\R^d))
=L^\infty(0,T;W) . \]
This proves assumption \eqref{en:A5} in the case $q>d$.

It is well known that the operator $-(-\varDelta)^{1/2}$ generates a bounded analytic 
semigroup $\{E^H(z)\}_{z\in\varSigma_{\pi/2}}$ with the kernel
\begin{equation}
E(z,x,y)=\frac{c_d}{z^d}\bigg(1+\frac{|x-y|^2}{z^2}\bigg)^{-\frac{d+1}{2}} \, ,
\end{equation}
where $c_d$ is a positive constant; see \cite[\textsection 1.1.3]{GHL14}. 
It is easy to check that the kernel $E(z,x,y)$ satisfies the condition of 
\cite[Proposition 2.9-(b)]{BK02} with $g(s)=C(1+s^2)^{-\frac{d+1}{2}}$. 
Hence, \cite[Proposition 2.9-(a), with $p=1$ and $q=\infty$ therein]{BK02} is satisfied. 
Substituting $s=2, p=1, q=\infty,  p_o=q\in(1,\infty)$ and $\varOmega=\R^d$ into 
\cite[Corollary 2.7]{BK02}, we see that the analytic semigroup generated by the 
operator $-(-\varDelta)^{1/2}$ is $R$-bounded on $L^q(\R^d)$ in the sector 
$\varSigma_\varphi$, for any $\varphi\in(0,\pi/2)$. Equivalently, in view of 
\cite[Theorem 4.2]{Weis2}, the family of operators 
$\{z(z+(-\varDelta)^{\frac12})^{-1}:z\in \varSigma_{\varphi+\pi/2}\}$ is $R$-bounded on 
$L^q(\R^d)$, for any $1<q<\infty$ and $\varphi\in(0,\pi/2)$. 
Then \cite[Theorems 4.1--4.2 and Remark 4.3]{KLL} implies that, for BDF methods up to order 6, 
\begin{equation}\label{MaxLp_frac_L2}
\begin{aligned}
&\frac{1}{\tau}\big\|(v_n-v_{n-1} )_{n=k}^N\big\|_{ L^p(L^2(\R^d))} 
+ \big\|(v_n )_{n=k}^N\big\|_{ L^p(H^1(\R^d))} \\
&\le C\Big(
\big\|(f_n)_{n=k}^N\big\|_{ L^p(L^2(\R^d))} 
+\frac{1}{\tau}\big\|(v_i)_{i=0}^{k-1}\big\|_{ L^p(L^2(\R^d))} 
+ \big\|(v_i )_{i=0}^{k-1}\big\|_{ L^p(H^1(\R^d))} \Big) . 
\end{aligned}
\end{equation}
and
\begin{equation}\label{MaxLp_frac_Lq}
\begin{aligned}
&\frac{1}{\tau}\big\|(v_n-v_{n-1} )_{n=k}^N\big\|_{ L^p(L^q(\R^d))} 
+ \big\|(v_n )_{n=k}^N\big\|_{ L^p(W^{1,q}(\R^d))} \\
&\le C\Big(
\big\|(f_n)_{n=k}^N\big\|_{ L^p(L^q(\R^d))} 
+\frac{1}{\tau}\big\|(v_i)_{i=0}^{k-1}\big\|_{ L^p(L^q(\R^d))} 
+ \big\|(v_i )_{i=0}^{k-1}\big\|_{ L^p(W^{1,q}(\R^d))} \Big) . 
\end{aligned}
\end{equation}
Estimates \eqref{MaxLp_frac_L2} and \eqref{MaxLp_frac_Lq} imply 
\eqref{en:A2}.

Overall, assumptions \eqref{en:A1}--\eqref{en:A5} are satisfied for  
$q\in(d,\infty)\cap [2,\infty)$.


\subsection{Proof of Proposition \ref{lem:framework4}}

Similarly, the operator $-\varDelta^2$ is self-adjoint and non-positive definite. Hence, it generates 
a bounded analytic semigroup of angle $\pi/2$ on the Hilbert space $H=L^2(\R^d)$, with 
domain $D_H=H^4(\R^d)$. This implies \eqref{en:A1}. 

Assumption \eqref{en:A3} is due to the time-independence and self-adjointness of the operator 
$\varDelta^2$. The self-adjointness also implies $\lambda=1$.

Since $W=H^2(\R^d)\cap W^{2,\infty}(\R^d)\hookrightarrow W^{2,q}(\R^d)$ for $q\ge 2$, 
if $w,v$ are bounded in $W$ with $\|w\|_{W}+\|v\|_{W}\le K$, 
then the local Lipschitz continuity of $f$ implies 
\begin{align*}
\|\varDelta f(w)-\varDelta f(v)\|_{X}
&=\|\varDelta f(w)-\varDelta f(v)\|_{L^2(\R^d)}
+\|\varDelta f(w)-\varDelta f(v)\|_{L^q(\R^d)} \\
&\le C_K(\|w-v\|_{H^2(\R^d)} +\|w-v\|_{W^{2,q}(\R^d)} ) \\
&\le C_K(\|w-v\|_{H^2(\R^d)} +\|w-v\|_{W^{2,\infty}(\R^d)}) 
= C_K\|w-v\|_{W} .
\end{align*}
This implies assumption \eqref{en:A4} in the case $q\ge 2$, with $\tilde\lambda=0$.

If $1<q<\infty$ and $2q>d$, then \cite[Theorems 5.2 and 5.9]{Adams} implies that, for 
any $\varepsilon>0$, there exists $C_\varepsilon>0$ satisfying 
\begin{align*}
&\|v\|_{H^2(\R^d)}  
\le \varepsilon \|v\|_{H^4(\R^d)}
+C_\varepsilon \|v\|_{L^2(\R^d)} 
&&\forall\, v\in H^{4}(\R^d) ,\\
&\|v\|_{W^{2,\infty}(\R^d)}  
\le \varepsilon \|v\|_{W^{4,q}(\R^d)}
+C_\varepsilon \|v\|_{L^q(\R^d)} 
&&\forall\, v\in W^{4,q}(\R^d) . 
\end{align*}
Since $D=H^{4}(\R^d)\cap W^{4,q}(\R^d)$ and 
$X=L^2(\R^d)\cap L^q(\R^d)$, the last two inequalities imply 
\eqref{Compact_interp} in assumption \eqref{en:A5}. 
Moreover, we have
\begin{equation*}
\begin{aligned}
&W^{1,p}(0,T;X)\cap L^p(0,T;D)  
\hookrightarrow 
W^{1,p}(0,T;L^2(\R^d))\cap L^p(0,T;H^4(\R^d))  \\
&\hookrightarrow 
W^{1-\theta,p}(0,T;H^{4\theta}(\R^d))
\qquad  \text{by using complex interpolation}\\
&\hookrightarrow 
L^\infty(0,T;H^2(\R^d)), 
\qquad\quad\,\, \text{when $(1-\theta)p>1$ and $4\theta>2$} .
\end{aligned}
\end{equation*} 
and 
\begin{equation*}
\begin{aligned}
&W^{1,p}(0,T;X)\cap L^p(0,T;D)  
\hookrightarrow 
W^{1,p}(0,T;L^q(\R^d))\cap L^p(0,T;W^{4,q}(\R^d))  \\
&\hookrightarrow 
W^{1-\theta,p}(0,T;W^{4\theta,q}(\R^d))
\quad \text{by using complex interpolation}\\
&\hookrightarrow 
L^\infty(0,T; W^{2,\infty}(\R^d)),
\qquad \text{when $(1-\theta)p>1$ and $(4\theta-2)q>d$.} 
\end{aligned}
\end{equation*} 
It remains to prove the existence of a $\theta$ satisfying the conditions above. 
In fact, if $1<p,q<\infty$ and $d/q+4/p<2$, then 
$\frac{1}{4}\big(2+\frac{d}{q}\big)<1-\frac{1}{p}$ and 
$\frac{1}{2}<1-\frac{1}{p}$. 
Hence, there exists $\theta\in(0,1)$ satisfying 
\[\frac{1}{4}\bigg(2+\frac{d}{q}\bigg)<\theta<1-\frac{1}{p}
\quad\text{and}\quad \frac{1}{2}<\theta<1-\frac{1}{p} .\]
Then $\theta$ satisfies $(1-\theta)p>1$, $4\theta>2$ and $(4\theta-2)q>d$. 
This proves \eqref{en:A5} in the case $q>\max(d/2,1)$. 

Now, according to \cite[Example 3.2 (A)]{HP97}, the semigroup generated by 
$-e^{\i\varphi}\varDelta^2$ satisfies a Gaussian estimate 
\[|E(te^{\i\varphi},x,y)|
\le \frac{C_\varphi}{t^{d/4}}\exp\bigg({-\frac{|x-y|^{4/3}}{C_\varphi\, t^{1/3}}}\bigg)\]
for any $\varphi\in(0,\pi/2)$. 
Substituting $s=2, p=1, q=\infty, p_o=q\in(1,\infty)$ and $\varOmega=\R^d$ into 
\cite[Corollary 2.7]{BK02}, we see that the analytic semigroup generated by the 
operator $-\varDelta^2$ is $R$-bounded on $L^q(\R^d)$, $1<q<\infty$, in the 
sector $\varSigma_\varphi$, for any $\varphi\in(0,\pi/2)$. 
Equivalently, in view of \cite[Theorems 4.2]{Weis2}, the family of operators 
$\{z(z+\varDelta^2)^{-1}:z\in \varSigma_{\varphi+\pi/2}\}$ is $R$-bounded on 
$L^q(\R^d)$, for any $1<q<\infty$ and $\varphi\in(0,\pi/2)$.
Then \cite[Theorems 4.1--4.2 and Remark 4.3]{KLL} implies that, for BDF methods 
of order up to 6, 
\begin{equation}\label{MaxLp_frac_L22}
\begin{aligned}
&\frac{1}{\tau}\big\|(v_n-v_{n-1} )_{n=k}^N\big\|_{ L^p(L^2(\R^d))} 
+ \big\|(v_n )_{n=k}^N\big\|_{ L^p(H^4(\R^d))} \\
&\le C\Big(
\big\|(f_n)_{n=k}^N\big\|_{ L^p(L^2(\R^d))} 
+\frac{1}{\tau}\big\|(v_i)_{i=0}^{k-1}\big\|_{ L^p(L^2(\R^d))} 
+ \big\|(v_i )_{i=0}^{k-1}\big\|_{ L^p(H^4(\R^d))} \Big) . 
\end{aligned}
\end{equation}
and
\begin{equation}\label{MaxLp_frac_Lq2}
\begin{aligned}
&\frac{1}{\tau}\big\|(v_n-v_{n-1} )_{n=k}^N\big\|_{ L^p(L^q(\R^d))} 
+ \big\|(v_n )_{n=k}^N\big\|_{ L^p(W^{4,q}(\R^d))} \\
&\le C\Big(
\big\|(f_n)_{n=k}^N\big\|_{ L^p(L^q(\R^d))} 
+\frac{1}{\tau}\big\|(v_i)_{i=0}^{k-1}\big\|_{ L^p(L^q(\R^d))} 
+ \big\|(v_i )_{i=0}^{k-1}\big\|_{ L^p(W^{4,q}(\R^d))} \Big) . 
\end{aligned}
\end{equation}
Estimates \eqref{MaxLp_frac_L22} and \eqref{MaxLp_frac_Lq2} imply 
\eqref{en:A2}.

Overall, assumptions \eqref{en:A1}--\eqref{en:A5} are satisfied for  
$q\in(d/2,\infty)\cap [2,\infty)$.

\subsection*{Acknowledgment}
The authors are grateful to Professor Christian
Lubich for stimulating discussions. 
The research stay of Buyang Li at Universit\"at T\" ubingen
was funded by the Alexander von Humboldt Foundation.

\bibliographystyle{amsplain}

\end{document}